\documentclass{my_aims}

\usepackage{amsmath}
\usepackage{paralist}
\usepackage{graphics}
\usepackage{epsfig}
\usepackage[colorlinks=true]{hyperref}

\hypersetup{urlcolor=blue, citecolor=red}
\usepackage{hyperref}

\usepackage{subfig}

\newtheorem{theorem}{Theorem}[section]

\theoremstyle{definition}

\title[Optimal control strategies for tuberculosis]{Optimal control strategies
for tuberculosis treatment: a case study in Angola}

\author[Cristiana J. Silva and Delfim F. M. Torres]{}



\email{cjoaosilva@ua.pt}
\email{delfim@ua.pt}

\thanks{The first author is supported by FCT post-doc grant SFRH/BPD/72061/2010}

\begin{document}

\maketitle

\centerline{\scshape Cristiana J. Silva and Delfim F. M. Torres}
\medskip
{\footnotesize
 \centerline{Center for Research and Development in Mathematics and Applications}
   \centerline{Department of Mathematics, University of Aveiro, 3810-193 Aveiro, Portugal}
}

% -----------------------------------------------------

\bigskip

\centerline{To Professor Helmut Maurer on the occasion of his 65th birthday.}

% -----------------------------------------------------

\begin{abstract}
We apply optimal control theory to a tuberculosis model
given by a system of ordinary differential equations.
Optimal control strategies are proposed to minimize
the cost of interventions. Numerical simulations
are given using data from Angola.
\end{abstract}

% -----------------------------------------------------

\section{Introduction}

Angola has one of the world's fastest growing economies.
It is a very rich country in terms of natural resources,
with oil, diamonds and hydroelectric power,
and a very fertile agricultural land.
However, Angola remains a third world country with
about one third of its population still depending
on subsistence agriculture \cite{url_tradingeconomics_Ang}.
In 2009, the life expectancy at birth was 52 years,
population median age was 17,
45\% of the population was under age 15 and only 4\% over 60 years,
and the total fertility rate per woman was 5.6.
Angola's total population was 10 million in 1990
and it is approximately double that figure today.
According to the Global Health Observatory Data Repository
of the World Health Organization (WHO),
in 2010 the infant mortality rate (probability of dying between birth and age 1)
was 98 per 1,000 live births; and the under-five mortality rate
(probability of dying by age 5) was 161 per 1,000 live births,
which corresponds to 121,000 deaths per year.

Tuberculosis (TB) is the leading cause of death among individuals infected with
the human immunodeficiency virus (HIV) and the rates of co-infection exceed
1,000 per 100,000 of total population in some countries in sub-Saharan Africa
\cite{Dye_etall_1999}. Between 1990 and 2005, TB incidence rates
tripled in African countries with high HIV prevalence \cite{avert:tb}.
Recent data from Luanda Sanatorium Hospital makes it clear that
TB is one of the leading causes of death in Angola
\cite{url_Min_Saude_Angola}. In 2010 alone, there were approximately
58,000 new TB cases in Angola. The country established
a National Program for the Control of Tuberculosis in 1981.
However, the civil war destroyed 70\% of Angola's health facilities,
leaving a substantial proportion of the population vulnerable to TB.
In 2007 the Ministry of Health confirmed that
the national program was implemented
in only 8.6\% of health units.
Moreover, the internal transport and distribution networks
for drugs are unreliable and 40\% of clinics
have experienced stockouts in TB drugs \cite{USAID_GOV}.
According to WHO's Global Tuberculosis Control Report of 2010, the figures
of TB in Angola are as follows: the induced mortality rate per 100,000 decreased
from 42 in 1990 to 30 in 2009; the incidence rate (including HIV)
increased from 205 in 1990 to 298 in 2009;
the new and relapse notification rate increased from 96 in 1990 to 223 in 2009;
the number of smear-positive cases increased from 3,804 in 2005 to 22,488 in 2009;
the number of extra-pulmonary TB cases increased from 266 in 1995 to 3,780 in 2010;
the relapse rate increased from 134 in 1995 to 2,444 in 2010; and the total number of retreated
individuals increased from 134 in 1995 to 7,776 in 2010. Treatment success rates
vary between 68\% in 2002 and 74\% in 2005, and are still within that range today,
well below WHO's target of 85\% \cite{USAID_GOV,WHOsite}.
The multidrug-resistant TB (resistant to, at least,
isoniazid and rifampicin) was about 2\% of all new cases in 2009.

TB remains a major health problem worldwide.
In 2011, the regional director of WHO for Africa stated that
despite the fact that member states have almost universally adopted
a strategy to stop tuberculosis, the implementation of this strategy
has been ineffective. Co-infections TB/HIV and TB multi-resistance
to drugs are some of the reasons for the
difficulty in controlling TB \cite{allafrica}.
Most cases of TB are caused by the \emph{Mycobacterium tuberculosis},
which is usually transmitted via airborne infection from someone who has active TB.
Approximately 10\% of infected people with \emph{Mycobacterium tuberculosis}
develop active TB disease, that is, approximately 90\% of infected people remain latent.
Latent infected TB people are asymptomatic and do not transmit TB, but may progress
to active TB through endogenous reactivation or exogenous reinfection \cite{Small_Fuj_2001,Styblo_1978}.
The anti-TB drugs developed since 1940 have helped to reduce the mortality rates significantly:
in clinical cases, cure rates of 90\% have been documented and in 2009 Angola's treatment
cure rate was about 72\% \cite{WHO_2011}. Three types of TB treatment are available:
vaccination to prevent infection; treatment to cure active TB;
and treatment of latent TB to prevent endogenous reactivation \cite{Gomes_etall_2007}.

Optimal control is a branch of mathematics that involves
finding optimal ways of controlling a dynamic system
\cite{Cesari_1983,Fleming_Rishel_1975,Pontryagin_et_all_1962,book:oc:Teo}.
While the usefulness of optimal control theory in epidemiology is nowadays well recognized
\cite{Maurer:2009,Maurer:2010,Maurer:2011,livro_Lenhart_2007,Rodrigues_Monteiro_Torres_2010,Rodrigues_Monteiro_Torres_Zinober_2011},
results pertaining to tuberculosis are a rarity \cite{SLenhart_2002},
and specific studies for the situation of Angola nonexistent.
Our aim in this paper is thus to use real data from Angola
to study optimal strategies for the minimization of the number
of active TB infectious and persistent latent individuals, taking into account
the cost of the measures for the treatment of these individuals.

The treatment of active infectious individuals
can take different amounts of time \cite{Kruk_etall_2008}.
Here we consider treatments with a duration of 6 months.
With these treatments one of the barriers to their success
is that the patients often do not complete them. Since after two months
patients no longer have symptoms of the disease and feel healed, many of them stop taking the medicines.
When the treatment plan is not completed, the patients are not cured and reactivation can occur and/or
the patients may develop resistent TB. The lack of support from family may cause
some patients to abandon treatment. Statistics from Luanda Sanatorium Hospital
show that 15 of the 270 TB patients abandoned the hospital during 2010.
A possible strategy to ensure patients complete the treatment involves supervision
and patient support. In fact, this is one the measures proposed
by WHO's Direct Observation Therapy (DOT) \cite{WHO_treatTB_2010}.
One example of treatment supervision consists of recording each dose of anti-TB drugs on the
patient treatment card \cite{WHO_treatTB_2010}. These measures are, however, very expensive,
since the patients must stay longer in the hospital or specialized people must be paid
to supervise patients until they finish their treatment. On the other hand, it is recognized
that the treatment of latent TB individuals reduces the chances of reactivation,
even though it is still unknown how treatment influences reinfection \cite{Gomes_etall_2007}.

In Section~\ref{sec:TBmodel} we consider a mathematical model for TB from \cite{Gomes_etall_2007},
which considers reinfection and post-exposure interventions. We alter this model by adding two control
functions: the first control $u_1$ is associated with preventive measures that help active infected patients
to complete the treatment; the second control $u_2$ governs the fraction of persistent
latent individuals that is put under treatment. Our aim is stated in precise terms
in Section~\ref{sec:opt:cont:prob}: to study how the two control measures
can reduce the number of infected and persistent latent individuals, taking into account
the cost associated with their implementation
(optimal control problem \eqref{modelTB_controls}--\eqref{mincostfunct}).
Our conclusion is that a successful implementation of these measures
can prevent the increase of new and reinfected TB cases.
In Section~\ref{sec:num:results} we present numerical simulations
for three different strategies and the corresponding
optimal control solutions are discussed. We end the paper
with some conclusions in Section~\ref{sec:conc}.

% -----------------------------------------------------

\section{TB model with controls}
\label{sec:TBmodel}

We study the mathematical model from \cite{Gomes_etall_2007},
where reinfection and post-exposure interventions for tuberculosis are considered.
To this model we add two control functions $u_1(\cdot)$ and $u_2(\cdot)$
and two real positive model constants $\epsilon_1$ and $\epsilon_2$.
The resulting model is given by the following system
of nonlinear ordinary differential equations:
\begin{equation}
\label{modelTB_controls}
\left\{
\begin{aligned}
\dot{S}(t) &= \mu N - \frac{\beta}{N} I(t) S(t) - \mu S(t)\\
\dot{L_1}(t) &= \frac{\beta}{N} I(t)\left( S(t)
+ \sigma L_2(t) + \sigma_R R(t)\right) - \left(\delta + \tau_1 + \mu\right)L_1(t)\\
\dot{I}(t) &= \phi \delta L_1(t) + \omega L_2(t) + \omega_R R(t)
- (\tau_0 + \epsilon_1 u_1(t) + \mu) I(t)\\
\dot{L_2}(t) &= (1 - \phi) \delta L_1(t) - \sigma \frac{\beta}{N} I(t) L_2(t)
- (\omega + \epsilon_2 u_2(t) + \tau_2 + \mu)L_2(t)\\
\dot{R}(t) &= (\tau_0 + \epsilon_1 u_1(t) )I(t) +  \tau_1 L_1(t)
+ (\tau_2 +\epsilon_2 u_2(t)) L_2(t) - \sigma_R \frac{\beta}{N} I(t) R(t)\\
&\quad \quad \quad - (\omega_R + \mu)R(t) \, .
\end{aligned}
\right.
\end{equation}
The population is divided into five categories (\textrm{i.e.},
control system \eqref{modelTB_controls} has five state variables):
\begin{itemize}
\item susceptible ($S$);
\item early latent ($L_1$), \textrm{i.e.}, individuals recently infected
(less than two years) but not infectious;
\item infected ($I$), \textrm{i.e.},
individuals who have active TB and are infectious;
\item persistent latent ($L_2$), \textrm{i.e.}, individuals who were infected and remain latent;
\item and recovered ($R$), \textrm{i.e.}, individuals who were previously infected
and have been treated.
\end{itemize}
The control $u_1$ represents the effort in preventing the failure of treatment
in active TB infectious individuals $I$, \textrm{e.g.}, supervising the patients,
helping them to take the TB medications regularly and to complete the TB treatment.
The control $u_2$ governs the fraction of persistent latent individuals $L_2$
that is put under treatment. The parameters $\epsilon_i \in (0, 1)$, $i=1, 2$,
measure the effectiveness of the controls $u_i$, $i=1, 2$, respectively, \textrm{i.e.},
these parameters measure the efficacy of treatment interventions for active and persistent
latent TB individuals, respectively.
We assume that the total population, $N$, with $N = S(t) + L_1(t) + I(t) + L_2(t) + R(t)$,
is constant in time. In other words, we assume
that the birth and death rates are equal.
By virtue of this assumption, we can reduce the control system
\eqref{modelTB_controls} from five to four state variables.
However, in this paper we prefer to keep the TB model
in form \eqref{modelTB_controls}, using relation
$S(t) + L_1(t) + I(t) + L_2(t) + R(t) = N$ as a test
to check the numerical results.

Following WHO's Global Health Observatory for Angola
in 2009 \cite{who_GHO_dataAngola}, we consider
$\mu = 1/52 \, yr^{-1}$, which corresponds to a life expectancy
at birth of 52 years. This and other values of the parameters
are given in Table~\ref{parameters}.
Moreover, following \cite{Gomes_etall_2007},
we assume that there are no disease-related deaths,
and that at birth all individuals are equally
susceptible and become different
as they experience infection and respective therapy.
The proportion of the population in each category changes
according to the system \eqref{modelTB_controls}.

Two different values for the initial conditions
$S(0)$, $L_1(0)$, $I(0)$, $L_2(0)$ and $R(0)$
are considered (see Table~\ref{statevar_initialvalues}):
\begin{itemize}
\item \textbf{Case 1}, initial values are borrowed from \cite{SLenhart_2002};
\item \textbf{Case 2}, initial values are based on the data from
Luanda's Sanatorium Hospital during September 2010 \cite{url_Min_Saude_Angola}.
\end{itemize}
The values of the rates $\delta$, $\phi$, $\omega$, $\omega_R$, $\sigma$ and
$\tau_0$ are taken from \cite{Gomes_etall_2007} and the references cited therein
(see Table~\ref{parameters} for the values of all the parameters). The parameter $\delta$
denotes the rate at which individuals leave the $L_1$ category; $\phi$ is the proportion
of individuals going to category $I$; $\omega$ is the rate of endogenous reactivation
for persistent latent infections (untreated latent infections); and $\omega_R$ is the rate
of endogenous reactivation for treated individuals
(for those who have undergone a therapeutic intervention).
The parameter $\sigma$ is a factor that measures the reduction in the risk of infection,
as a result of acquired immunity to a previous infection,
for persistent latent individuals, \textrm{i.e.}, this factor affects
the rate of exogenous reinfection of untreated individuals;
while $\sigma_R$ represents the same parameter factor but for treated patients.
In our simulations we consider the case where the susceptibility to reinfection
of treated individuals equals that of latents: $\sigma_R = \sigma$.
The parameter $\tau_0$ is the rate of recovery under treatment of active TB
(assuming an average infectiousness duration of six months).
The parameters $\tau_1$ and $\tau_2$ apply to latent individuals $L_1$ and $L_2$,
respectively, and are the rates at which chemotherapy or a post-exposure vacine is applied.
In \cite{Gomes_etall_2007}, different values for these rates are considered:
the case where no treatment of latent infections occur ($\tau_1 = \tau_2 = 0$);
the case where there is an immediate treatment of persistent latent infections ($\tau_2 \to \infty$);
and the case where there is a moderate treatment of persistent latent infections ($\tau_2 \in [0.1,1]$).
The first and second cases are not interesting from the optimal control point of view.
In our paper we consider, without loss of generality, that the rate of recovery of early latent
individuals under post-exposure interventions is equal to the rate of recovery under treatment of active TB,
$\tau_1 = 2 \, yr^{-1}$, and greater than the rate of recovery of persistent latent individuals
under post-exposure interventions, $\tau_2 = 1 \, yr^{-1}$.
It is assumed that the rate of infection for susceptible individuals is proportional
to the number of infectious individuals and the constant of proportionality is $\beta$,
which is the transmission coefficient.

From the epidemiologistic point of view, an important quantity
is given by the basic reproduction number $R_0$. This is
the expected number of secondary cases produced in a complete susceptible population,
by a typical infected individual during his/her entire period of infectiousness
\cite{Rodrigues_Monteiro_Torres_Zinober_2011}.
In the absence of controls ($u_1 = u_2 \equiv 0$), it is known (see \cite{Gomes_etall_2007})
that the basic reproduction number for system \eqref{modelTB_controls}
is proportional to the transmission coefficient $\beta$:
\begin{equation*}
R_0 = \beta \frac{\delta (\omega + \phi \mu)(\omega_R
+ \mu)}{\mu(\omega_R + \tau_0 + \mu)(\delta + \mu)(\omega + \mu)}.
\end{equation*}
The endemic threshold is given at $R_0 = 1$ and indicates the minimal transmission
potential that sustains endemic disease, \textrm{i.e.}, when $R_0 < 1$
the disease will die out and for $R_0 > 1$ the disease may become endemic.
We can prove that variations of the parameter $\beta$ are associated
with different values for the optimal control $u_1$ but do not interfere
significantly with the optimal control $u_2$. Moreover, as expected,
as $\beta$ decreases the fraction of latent and infected individuals
also decreases. Here we consider $\beta = 100$, which corresponds
to the endemic case ($R_0 = 2.96 > 1$).
The measures of control efficacy
are fixed to the values $\epsilon_i = 1/2$, $i=1, 2$.
Because the risk of developing disease after
infection is much higher in the first five years following infection,
and declines exponentially after that \cite[Sec.~5.5]{Styblo_1991},
we take the total simulation time duration as $T = 5$ years.

% -----------------------------------------------------

\section{The optimal control problem}
\label{sec:opt:cont:prob}

Consider the state system \eqref{modelTB_controls} of ordinary differential
equations in $\mathbb{R}^5$ with the set of admissible control functions given by
\begin{equation*}
\Omega = \{ (u_1(\cdot), u_2(\cdot)) \in (L^{\infty}(0, T))^2
\, | \,  0 \leq u_1 (t), u_2(t) \leq 1 ,  \, \forall \, t \in [0, T] \, \} .
\end{equation*}
The objective functional is given by
\begin{equation}
\label{costfunction}
J(u_1(\cdot), u_2(\cdot)) = \int_0^{T} \left[ I(t) + L_2(t)
+ \tfrac{1}{2} W_1 u_1^2(t) + \tfrac{1}{2} W_2 u_2^2(t) \right] dt \, ,
\end{equation}
where the constants $W_1$ and $W_2$ are a measure
of the relative cost of the interventions
associated with the controls $u_1$ and $u_2$, respectively.
In our numerical simulations we consider the weight $W_1$
associated with the control $u_1$ to be greater than the weight $W_2$ associated
with $u_2$: $W_1 = 500$ and $W_2 = 50$.
The reason for this is that the cost associated with $u_1$ includes
the cost of holding active infected patients $I$ in the hospital
or paying professionals to supervise them, ensuring that they finish
their treatment, which is costly to implement.
In Section~\ref{sec:same} we compare this situation to the one when
the weights take the same values: $W_1 = W_2 = 50$.
We consider the optimal control problem of determining
$\left(S^*(\cdot), L_1^*(\cdot), I^*(\cdot), L_2^*(\cdot), R^*(\cdot)\right)$,
associated with an admissible control pair
$\left(u_1^*(\cdot), u_2^*(\cdot) \right) \in \Omega$ on the time interval $[0, T]$,
satisfying \eqref{modelTB_controls}, the initial conditions
$S(0)$, $L_1(0)$, $I(0)$, $L_2(0)$ and $R(0)$ (see Table~\ref{statevar_initialvalues}),
and minimizing the cost functional \eqref{costfunction}, \textrm{i.e.},
\begin{equation}
\label{mincostfunct}
J(u_1^*(\cdot), u_2^*(\cdot))
= \min_{\Omega} J(u_1(\cdot), u_2(\cdot)) \, .
\end{equation}

\begin{theorem}
\label{theorem_optcont}
The problem \eqref{modelTB_controls}, \eqref{mincostfunct}
with fixed initial conditions $S(0)$, $L_1(0)$, $I(0)$, $L_2(0)$ and $R(0)$,
and fixed final time $T$, admits a unique optimal solution
$\left(S^*(\cdot), L_1^*(\cdot), I^*(\cdot), L_2^*(\cdot), R^*(\cdot)\right)$
associated with an optimal control pair $\left(u_1^*(\cdot), u_2^*(\cdot)\right)$ on $[0, T]$.
Moreover, there exist adjoint functions, $\lambda_1^*(\cdot)$, $\lambda_2^*(\cdot)$,
$\lambda_3^*(\cdot)$, $\lambda_4^*(\cdot)$ and $\lambda_5^*(\cdot)$, such that
\begin{equation}
\label{adjoint_function}
\left\{
\begin{split}
\dot{\lambda^*_1}(t) &= \lambda^*_1(t) \left(\frac{\beta}{N} I^*(t)
+ \mu \right) - \lambda^*_2(t) \frac{\beta}{N} I^*(t)\\
\dot{\lambda^*_2}(t) &= \lambda^*_2(t)(\delta + \tau_1 + \mu)
- \lambda^*_3(t) \phi \delta - \lambda^*_4(t) (1 - \phi) \delta
- \lambda^*_5(t)\tau_1\\
\dot{\lambda^*_3}(t) &= -1 + \lambda^*_1(t) \frac{\beta}{N} S^*(t)
- \lambda^*_2(t) \frac{\beta}{N}(S^*(t) + \sigma L_2^*(t) + \sigma_R R^*(t))\\
+ \lambda^*_3&(t) (\tau_0 +\epsilon_1 u_1^*(t) + \mu)
+\lambda^*_4(t)\sigma \frac{\beta}{N} L_2^*(t) - \lambda^*_5(t)\left(\tau_0
+ \epsilon_1 u^*_1(t) - \sigma_R \frac{\beta}{N} R^*(t) \right)\\
\dot{\lambda^*_4}(t) &= - 1 - \lambda^*_2(t) \frac{\beta}{N}I^*(t) \sigma
- \lambda^*_3(t) \omega + \lambda^*_4(t)\left(\sigma \frac{\beta}{N} I^*(t)
+ \omega + \epsilon_2 u^*_2(t) + \tau_2 + \mu\right)\\
- \lambda^*_5&(t)\left( \tau_2 + \epsilon_2 u^*_2(t) \right)\\
\dot{\lambda^*_5}(t) &= -\lambda^*_2(t) \sigma_R \frac{\beta}{N}I^*(t)
- \lambda^*_3(t) \omega_R + \lambda^*_5(t)\left(\sigma_R
\frac{\beta}{N} I^*(t)  +\omega_R + \mu\right)   \, ,
\end{split}
\right.
\end{equation}
with transversality conditions
$\lambda^*_i(T) = 0$, $i=1, \ldots, 5$.
Furthermore,
\begin{equation}
\label{optcontrols}
\begin{split}
u_1^*(t) &= \min \left\{ \max \left\{0, \frac{\epsilon_1 I^*
(\lambda^*_3 - \lambda^*_5)}{W_1}\right\}, 1 \right\} \, ,\\
u_2^*(t) &= \min \left\{ \max \left\{0, \frac{\epsilon_2
L^*_2(\lambda^*_4 - \lambda^*_5)}{W_2}\right\}, 1 \right\}  \, .
\end{split}
\end{equation}
\end{theorem}

\begin{proof}
The existence of optimal controls $\left(u_1^*(\cdot), u_2^*(\cdot)\right)$
and associated optimal solution $\left(S^*, L_1^*, I^*, L_2^*, R^*\right)$
comes from the convexity of the integrand of the cost function \eqref{costfunction}
with respect to the controls $(u_1, u_2)$ and the Lipschitz property
of the state system with respect to state variables $\left(S, L_1, I, L_2, R\right)$
(see, \textrm{e.g.}, \cite{Cesari_1983,Fleming_Rishel_1975}
for existence results of optimal solutions).
According to the Pontryagin Maximum Principle \cite{Pontryagin_et_all_1962},
if $(u_1^*(\cdot), u_2^*(\cdot)) \in \Omega$ is optimal for the problem
\eqref{modelTB_controls}, \eqref{mincostfunct} with the initial conditions
given in Table~\ref{statevar_initialvalues} and fixed final time $T$,
then there exists a nontrivial absolutely continuous mapping $\lambda : [0, T] \to \mathbb{R}^5$,
$\lambda(t) = \left(\lambda_1(t), \lambda_2(t), \lambda_3(t), \lambda_4(t), \lambda_5(t)\right)$,
called the \emph{adjoint vector}, such that
\begin{equation*}
\left\{
\begin{split}
\dot{S} &= \frac{\partial H}{\partial \lambda_1} \, ,\\
\dot{L}_1 &= \frac{\partial H}{\partial \lambda_2} \, ,\\
\dot{I} &= \frac{\partial H}{\partial \lambda_3} \, ,\\
\dot{L}_2 &= \frac{\partial H}{\partial \lambda_4} \, ,\\
\dot{R} &= \frac{\partial H}{\partial \lambda_5}
\end{split}
\right.
\end{equation*}
and
\begin{equation}
\label{adjsystemPMP}
\begin{cases}
\dot{\lambda}_1 = -\frac{\partial H}{\partial S} \, ,\\
\dot{\lambda}_2 = -\frac{\partial H}{\partial L_1} \, ,\\
\dot{\lambda}_3 = -\frac{\partial H}{\partial I} \, ,\\
\dot{\lambda}_4 = -\frac{\partial H}{\partial L_2} \, ,\\
\dot{\lambda}_5 = -\frac{\partial H}{\partial R} \, ,
\end{cases}
\end{equation}
where the Hamiltonian $H$ is defined by
\begin{equation*}
\begin{split}
H&= H(S, L_1, I, L_2, R, \lambda, u_1, u_2) \\
&=I + L_2  + \frac{W_1}{2}u_1^2 + \frac{W_2}{2}u_2^2 \\
&\quad + \lambda_1 \left(\mu N - \frac{\beta}{N} I S - \mu S \right)\\
&\quad  + \lambda_2 \left( \frac{\beta}{N} I\left( S
+ \sigma L_2 + \sigma_R R\right) - (\delta + \tau_1 + \mu)L_1 \right)\\
&\quad  + \lambda_3 \left(\phi \delta L_1 + \omega L_2
+ \omega_R R - (\tau_0 + \epsilon_1 u_1 + \mu) I \right)\\
&\quad  + \lambda_4 \left((1 - \phi) \delta L_1
- \sigma \frac{\beta}{N} I L_2
- (\omega + \epsilon_2 u_2 + \tau_2 + \mu)L_2 \right)\\
&\quad   + \lambda_5 \left((\tau_0 + \epsilon_1 u_1 )I
+ \tau_1 L_1 + \left(\tau_2 +\epsilon_2 u_2\right) L_2
- \sigma_R \frac{\beta}{N} I R - (\omega_R + \mu)R \right)
\end{split}
\end{equation*}
and the minimization condition
\begin{multline}
\label{maxcondPMP}
H\left(S^*(t), L_1^*(t), I^*(t), L_2^*(t), R^*(t), \lambda^*(t), u_1^*(t), u_2^*(t)\right)\\
= \min_{0 \leq u_1, u_2 \leq 1}
H\left(S^*(t), L_1^*(t), I^*(t), L_2^*(t), R^*(t), \lambda^*(t), u_1, u_2\right)
\end{multline}
holds almost everywhere on $[0, T]$. Moreover, the transversality conditions
$\lambda_i(T) = 0$, $i =1,\ldots, 5$, hold.
System~\eqref{adjoint_function} is derived from \eqref{adjsystemPMP},
and the optimal controls \eqref{optcontrols} come from
the minimization condition \eqref{maxcondPMP}.
For small final time $T$, the optimal control pair given by \eqref{optcontrols}
is unique due to the boundedness of the state and adjoint functions
and the Lipschitz property of systems \eqref{modelTB_controls}
and \eqref{adjoint_function} (see, \textrm{e.g.}, \cite{SLenhart_2002}).
\end{proof}

% -----------------------------------------------------

\section{Numerical results and discussion}
\label{sec:num:results}

In this section we present different optimal control strategies
for the TB model \eqref{modelTB_controls}
under the parameter values given in Table~\ref{parameters}.
Three control strategies are explored:
\begin{itemize}
\item \textbf{Strategy 1}, which implements measures that prevent the failure
of treatment in active TB individuals $I$ (control $u_1$ alone);

\item \textbf{Strategy 2}, which considers treatment
of persistent latent individuals $L_2$ with anti-TB drugs (control $u_2$ alone);

\item \textbf{Strategy 3}, which implements measures
for preventing treatment failure in active TB
individuals and variation of the fraction of persistent latent individuals
that are put under treatment (controls $u_1$ and $u_2$).
\end{itemize}
% ------------------------------------------------
\begin{table}
\centering
\begin{tabular}{|l|l|l|}
\hline
{\small{Symbol}} & {\small{Description}}  & {\small{Value}} \\
\hline
{\small{$\beta$}} & {\small{Transmission coefficient}}  & {\small{$100$}}\\
{\small{$\mu$}} & {\small{Death and birth rate}}  & {\small{$1/52 \, yr^{-1}$}}\\
{\small{$\delta$}} & {\small{Rate at which individuals leave $L_1$}}  & {\small{$12 \, yr^{-1}$}}\\
{\small{$\phi$}} & {\small{Proportion of individuals going to $I$}}  & {\small{$0.05$}}\\
{\small{$\omega$}} & {\small{Endogenous reactivation rate for persistent latent infections}}  & {\small{$0.0002 \, yr^{-1}$}}\\
{\small{$\omega_R$}} & {\small{Endogenous reactivation rate for treated individuals}}   &{\small{$0.00002 \, yr^{-1}$}}\\
{\small{$\sigma$}} & {\small{Factor reducing the risk of infection as a result of acquired}}  & \\
& {\small{immunity to a previous infection for $L_2$}} & {\small{$0.25$}} \\
{\small{$\sigma_R$}} & {\small{Rate of exogenous reinfection of treated patients}}  & {\small{0.25}} \\
{\small{$\tau_0$}} & {\small{Rate of recovery under treatment of active TB}}  &  {\small{$2 \, yr^{-1}$}}\\
{\small{$\tau_1$}} & {\small{Rate of recovery under treatment of latent individuals $L_1$}}  &  {\small{$2 \, yr^{-1}$}}\\
{\small{$\tau_2$}} & {\small{Rate of recovery under treatment of latent individuals $L_2$}}  &  {\small{$1 \, yr^{-1}$}}\\
{\small{$N$}} & {\small{Total population}} & {\small{$30,000$}} \\
{\small{$T$}} & {\small{Total simulation duration}} & {\small{$5$ $yr$}} \\
{\small{$\epsilon_1$}} & {\small{Efficacy of treatment of active TB $I$}} & {\small{$0.5$}} \\
{\small{$\epsilon_2$}} & {\small{Efficacy of treatment of latent TB $L_2$}} & {\small{$0.5$}} \\
{\small{$W_1$}} & {\small{Weight constant of control $u_1$}} & {\small{$500$}}\\
{\small{$W_2$}} & {\small{Weight constant of control $u_2$}} & {\small{$50$}}\\ \hline
\end{tabular}
\caption{Parameter values for the optimal control problem \eqref{modelTB_controls}--\eqref{mincostfunct}}
\label{parameters}
\end{table}
% ------------------------------------------------

Different approaches were used to obtain and confirm the numerical results.
One approach consisted in using IPOPT \cite{IPOPT}
and the algebraic modeling language AMPL \cite{AMPL}.
A second approach was to use the PROPT Matlab Optimal Control Software \cite{PROPT}.
The results coincide with the ones obtained by an iterative method
that consists in solving the system of ten ODEs given by
\eqref{modelTB_controls} and \eqref{adjoint_function}.
For that, first one solves system \eqref{modelTB_controls}
with a guess for the controls over the time interval
$[0, T]$ using a forward fourth-order Runge--Kutta scheme and the transversality
conditions $\lambda_i(T) = 0$, $i=1, \ldots, 5$. Then,
system \eqref{adjoint_function} is solved by
a backward fourth-order Runge--Kutta scheme using the current
iteration solution of \eqref{modelTB_controls}.
The controls are updated by using a convex combination of the previous controls
and the values from \eqref{optcontrols}. The iteration is stopped if the values of unknowns
at the previous iteration are very close to the ones at the present iteration.
For more details see, \emph{e.g.}, \cite{SLenhart_2002}.

To begin, we study the influence of the initial values of the state
variables, as given in Table~\ref{statevar_initialvalues},
on the optimal solution associated with Strategy~3.
% ------------------------------------------------
\begin{table}
\centering
\begin{tabular}{|l|l|l|l|l|}
\hline
{\small{States}}  & {\small{Case 1}} & {\small{Case 1}}  & {\small{Case 2}} & {\small{Case 2}}\\
{\small{(initial value)}} & {\small{(fraction)}}  & {\small{($\%$)}}  & {\small{(fraction)}}  & {\small{($\%$)}} \\
\hline
{\small{$S(0)$}} &  {\small{$(76/120)N$}} & {\small{$\simeq 63.3\%$}} & {\small{$(16/50)N$}} & {\small{$32\%$}} \\
\hline
{\small{$L_1(0)$}} & {\small{$(36/120)N$}} & {\small{$30\% $}} & {\small{$(28/50)N$}}  & {\small{$56\%$}} \\
\hline
{\small{$I(0)$}} & {\small{$(5/120)N$}} & {\small{$\simeq 4.2\% $}} & {\small{$(3/50)N$}} & {\small{$6\%$}}\\
\hline
{\small{$L_2(0)$}} & {\small{$(2/120)N$}} & {\small{$\simeq 1.7\% $}} & {\small{$(2/50)N$}} & {\small{$4\%$}} \\
\hline
{\small{$R(0)$}} & {\small{$(1/120)N$}} & {\small{$\simeq 0.8\% $}} & {\small{$(1/50)N$}} & {\small{$2\%$}} \\
\hline
\end{tabular}
\caption{Initial values for the state variables of \eqref{modelTB_controls}}
\label{statevar_initialvalues}
\end{table}

% -----------------------------------------------------

\subsection{Different initial values (Case 1 and Case 2) for the state variables}

Our aim is to compare the optimal solution associated with Strategy~3
(when both controls $u_1$ and $u_2$ are considered)
under different initial values for the state variables:
in the first case we consider the initial values of \cite{SLenhart_2002}
(Case~1) and in the second case we consider the initial values given in
an official communication from Luanda Sanatorium Hospital
\cite{url_Min_Saude_Angola} (Case~2).
% ------------------------------------------------
\begin{figure}
  \begin{center}
  \subfloat[\footnotesize{Optimal Control $u_1^*$}]{\label{fig:comp:dadosiniciais:u1}
  \includegraphics[width=0.45\textwidth]{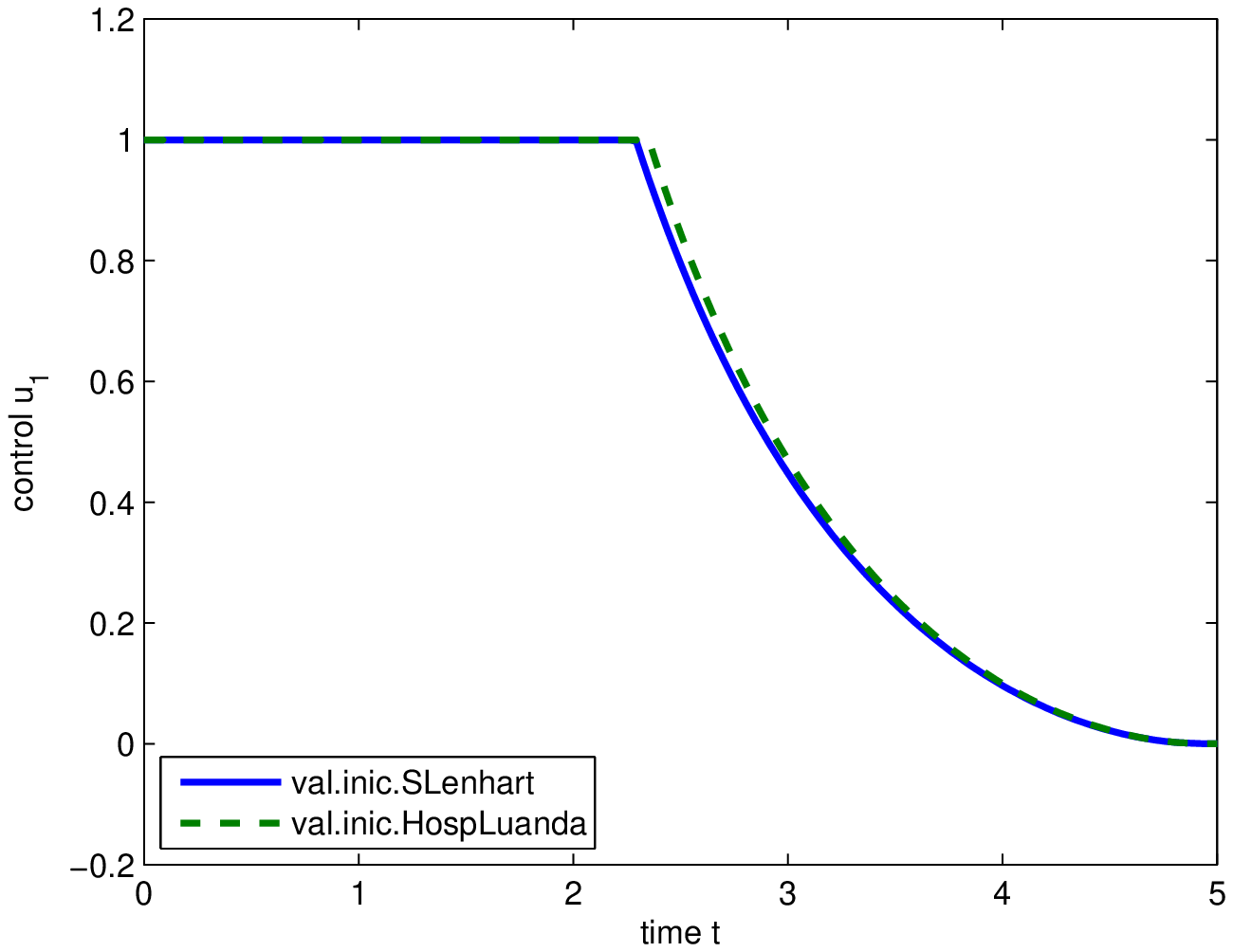}}
  \subfloat[\footnotesize{Optimal Control $u_2^*$}]{\label{fig:comp:dadosiniciais:u2}
  \includegraphics[width=0.45\textwidth]{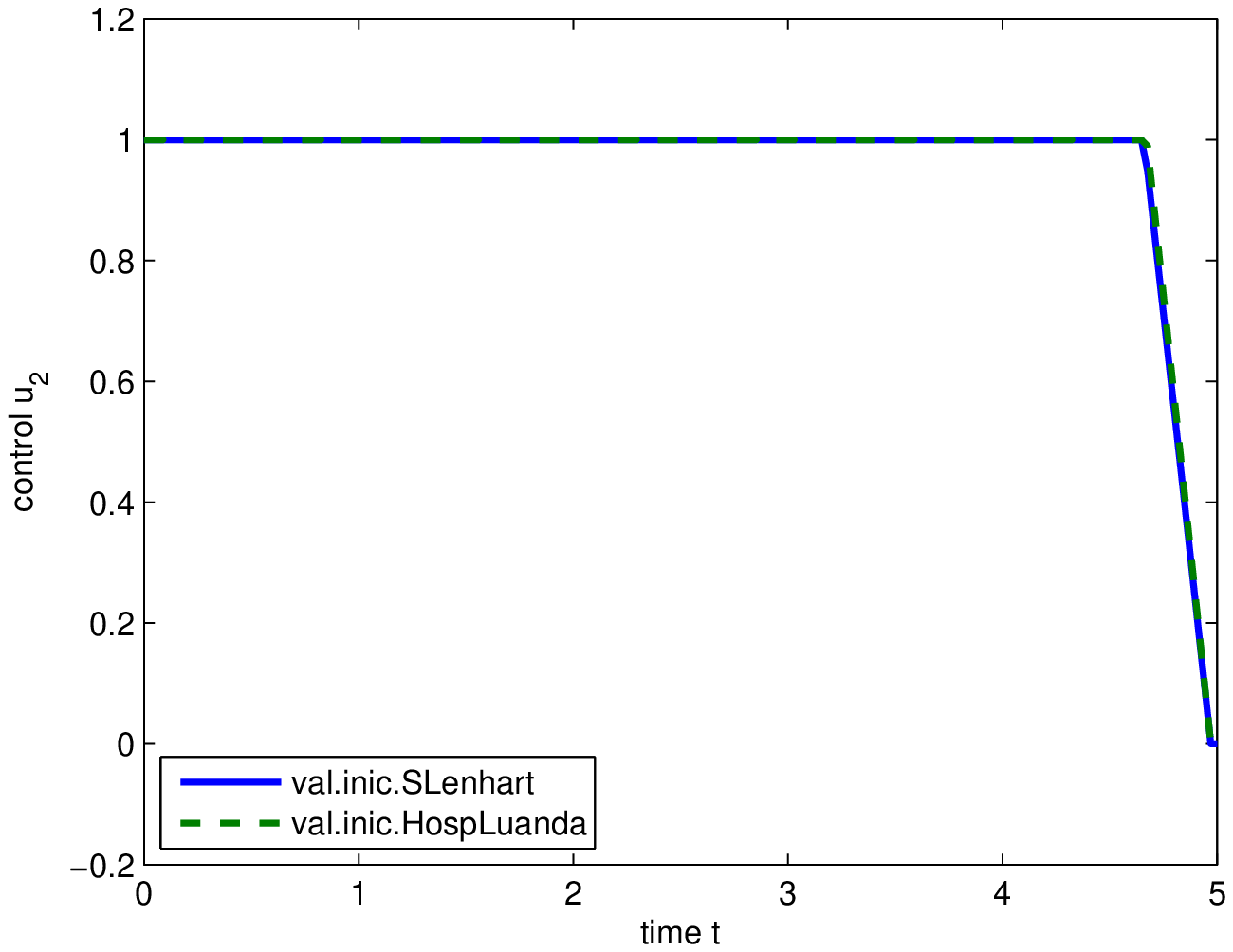}}
  \caption{Optimal control functions in Case 1 (solid line) and Case 2 (dashed line),
  in a time scale in units of years, for Strategy~3}
  \label{fig:variar:dados:iniciais:controlos}
  \end{center}
\end{figure}
% ------------------------------------------------
In Figure~\ref{fig:variar:dados:iniciais:controlos} we observe that
the differences between the initial values considered in Cases~1 and 2
(Table~\ref{statevar_initialvalues}) do not
influence the optimal controls significantly.
% ------------------------------------------------
\begin{figure}
\begin{center}
\subfloat[\footnotesize{$(I^*+L_2^*)/N$}]{
  \includegraphics[width=0.45\textwidth]{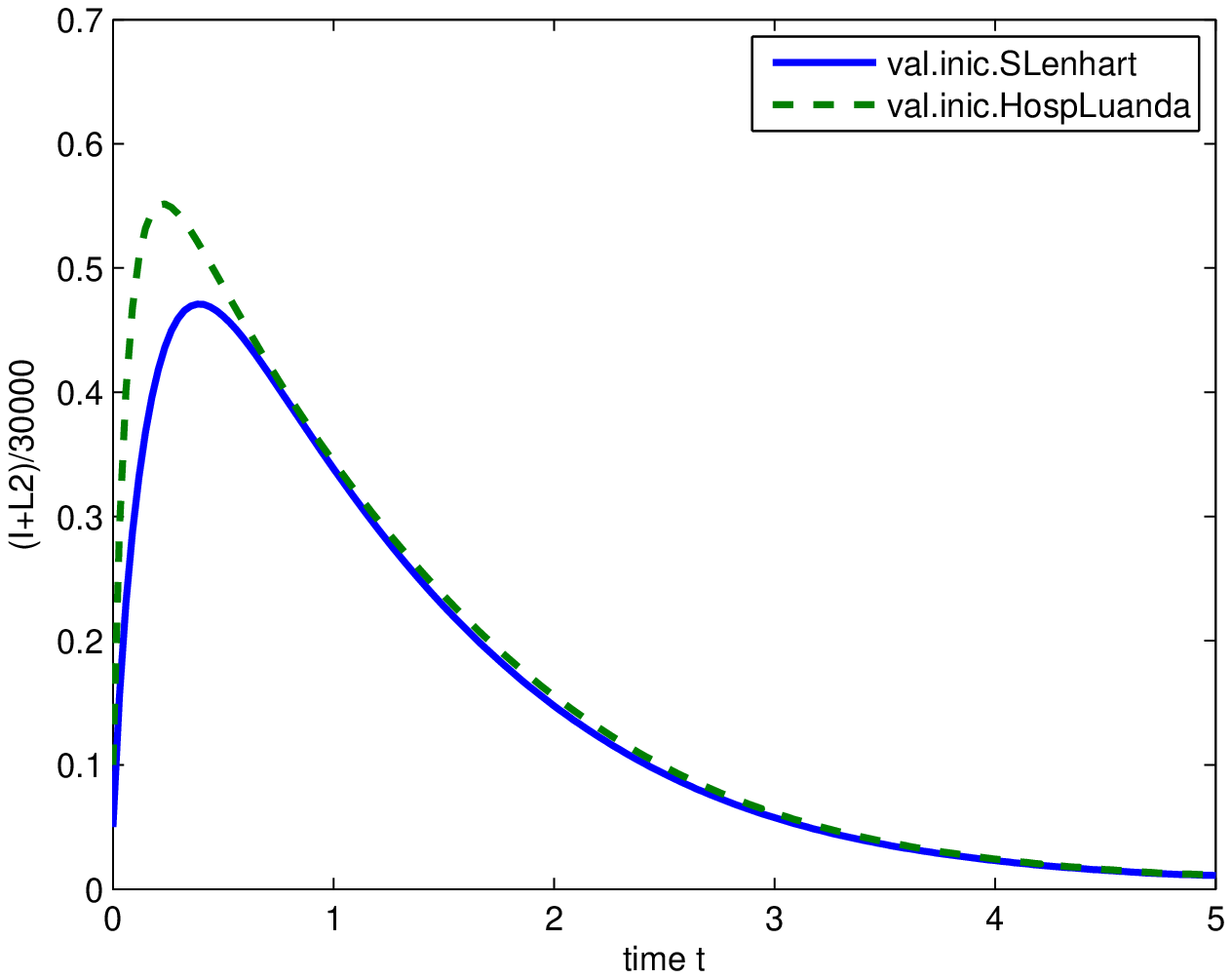}}
\subfloat[\footnotesize{$S^*/N$}]{
  \includegraphics[width=0.45\textwidth]{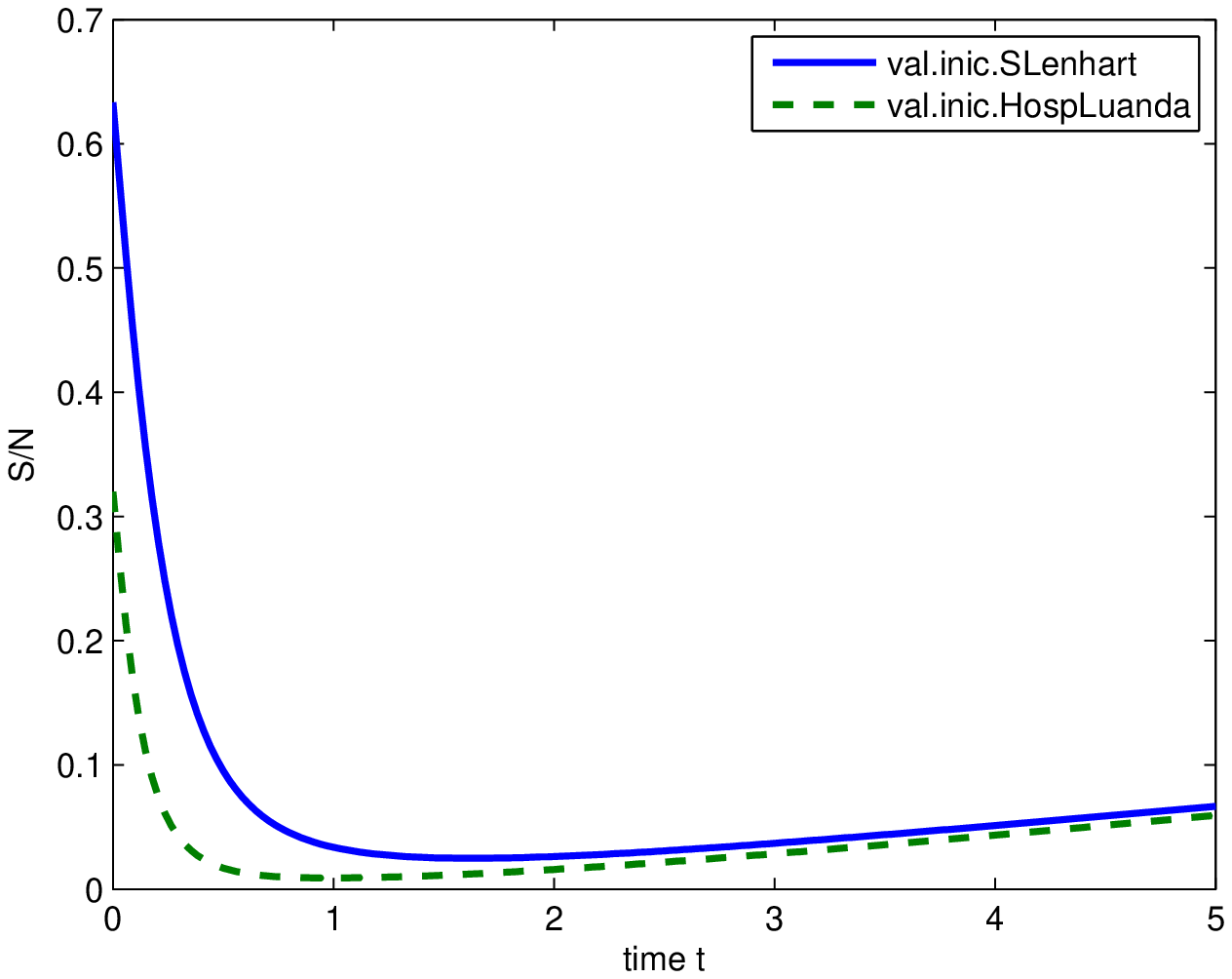}}
\newline
\subfloat[\footnotesize{$L_1^*/N$}]{\label{fig:comp:dadosiniciais:L1}
  \includegraphics[width=0.45\textwidth]{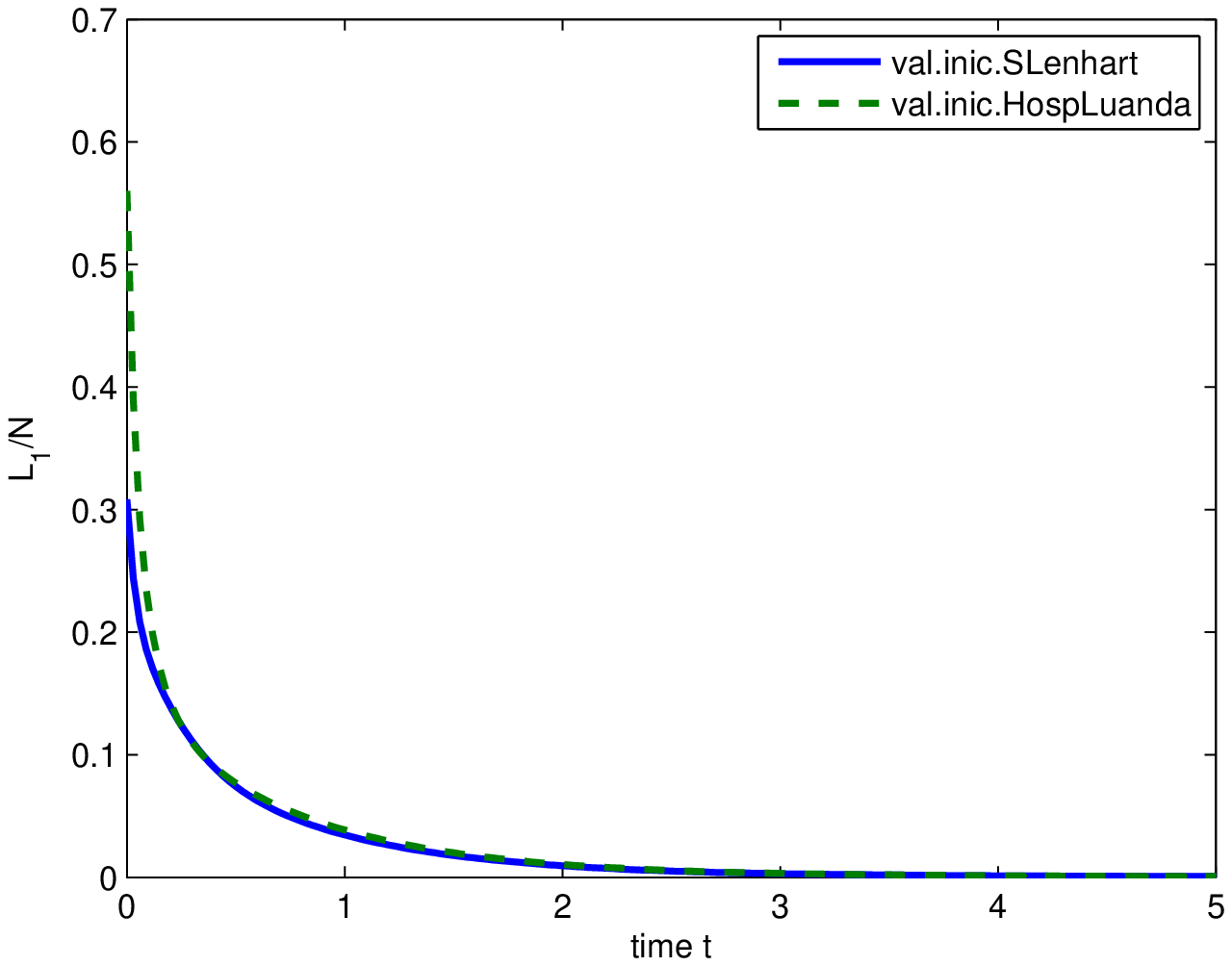}}
\subfloat[\footnotesize{$R^*/N$}]{\label{fig:comp:dadosiniciais:R}
  \includegraphics[width=0.45\textwidth]{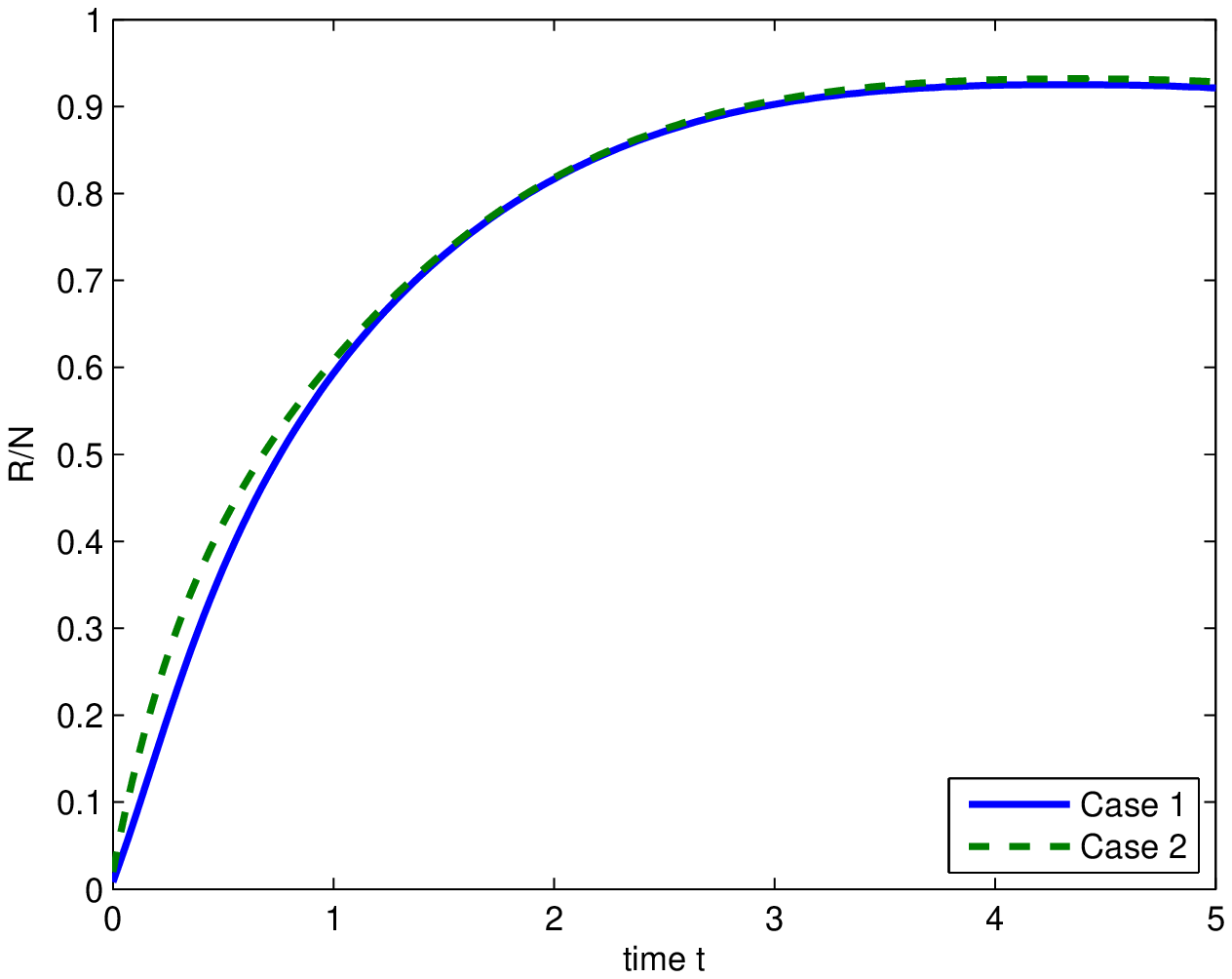}}
\caption{Optimal states in Case 1 (solid line) and Case 2 (dashed line),
in a time scale in units of years, for Strategy~3}
\label{fig:variar:dados:iniciais2}
\end{center}
\end{figure}
% ------------------------------------------------
In Case~1, the initial percentage of susceptible individuals $S(0)$
is approximately twice the respective percentage in Case~2.
However, with the optimal controls $\left(u_1^*, u_2^*\right)$
illustrated in Figure~\ref{fig:variar:dados:iniciais:controlos}, after
two years this difference is substantially reduced, as we can observe
in Figure~\ref{fig:variar:dados:iniciais2}. With respect to the initial
values of early latent individuals $L_1(0)$, the percentage in Case~1
is 30\% and in Case~2 it is 56\%. Here, in both cases,
the percentage of early latent individuals is reduced
to approximately 10\% after approximately
140 days. We observe that even if in Case~2
the initial fraction of early latent individuals is much higher than in Case~1,
the optimal controls ensure a big reduction in the fraction of
early latent individuals and, after no more than 5 months, the difference
between the two cases no longer exists. Given these results,
from now on we consider, without loss of generality,
Case~1 for the initial values of the state variables.
Our numerical simulations were checked
using the relation $S(t) + L_1(t) + I(t) + L_2(t) + R(t) = N$.
Figure~\ref{fig:total:pop:N} shows that $S(t) + L_1(t) + I(t) + L_2(t) + R(t)$
is constant along time for Strategy~3, as expected.
% -----------------------------------------------------
\begin{figure}[!htb]
\begin{center}
\includegraphics[width=0.45\textwidth]{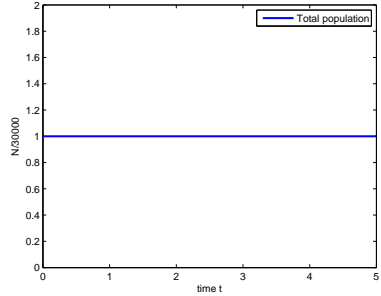}
\caption{Plot of $\left(S(t) + L_1(t) + I(t) + L_2(t) + R(t)\right)/N$
versus $t$ in years for Strategy~3}
\label{fig:total:pop:N}
\end{center}
\end{figure}

% -----------------------------------------------------

\subsection{Strategy~1 versus Strategy~3}
\label{sub:sec:st1}

If we choose to apply only the control $u_1$ (Strategy~1)
and compare this situation to the case where both controls $u_1$ and $u_2$
are applied (Strategy~3), we observe that the optimal control $u_1^*$ stays
at the upper bound for almost the same duration in both situations (see Figure~\ref{fig:so:u1}).
% ------------------------------------------------
\begin{figure}
\begin{center}
\subfloat[\footnotesize{Control $u_1^*$}]{\label{fig:u1:so:u1}
  \includegraphics[width=0.45\textwidth]{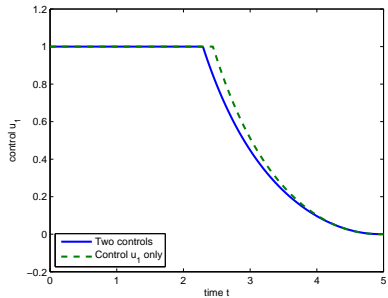}}
\caption{The optimal control $u_1^*$ for Strategy~1 (dashed line)
and Strategy~3 (solid line) in a time scale in units of years}
\label{fig:so:u1}
\end{center}
\end{figure}
% ------------------------------------------------
When only the control $u_1$ is considered, the fraction of persistent latent individuals $L_2$
is greater than the respective fraction associated with Strategy~3, for the entire five years
(see Figure~\ref{fig:so:u1:statevar} A and E). This implies a higher value of the cost
functional \eqref{costfunction} associated with Strategy~1
when compared to the cost associated with Strategy~3.
When one compares the change in the fraction of susceptible ($S$), early latent ($L_1$)
and infected ($I$) individuals, no difference is observed  between Strategies~1 and 3
(see Figure~\ref{fig:so:u1:statevar} B, C and D).
% ------------------------------------------------
\begin{figure}
\begin{center}
\subfloat[\footnotesize{$(I^* + L_2^*)/N$}]{\label{fig:so:u1:I:L2}
  \includegraphics[width=0.45\textwidth]{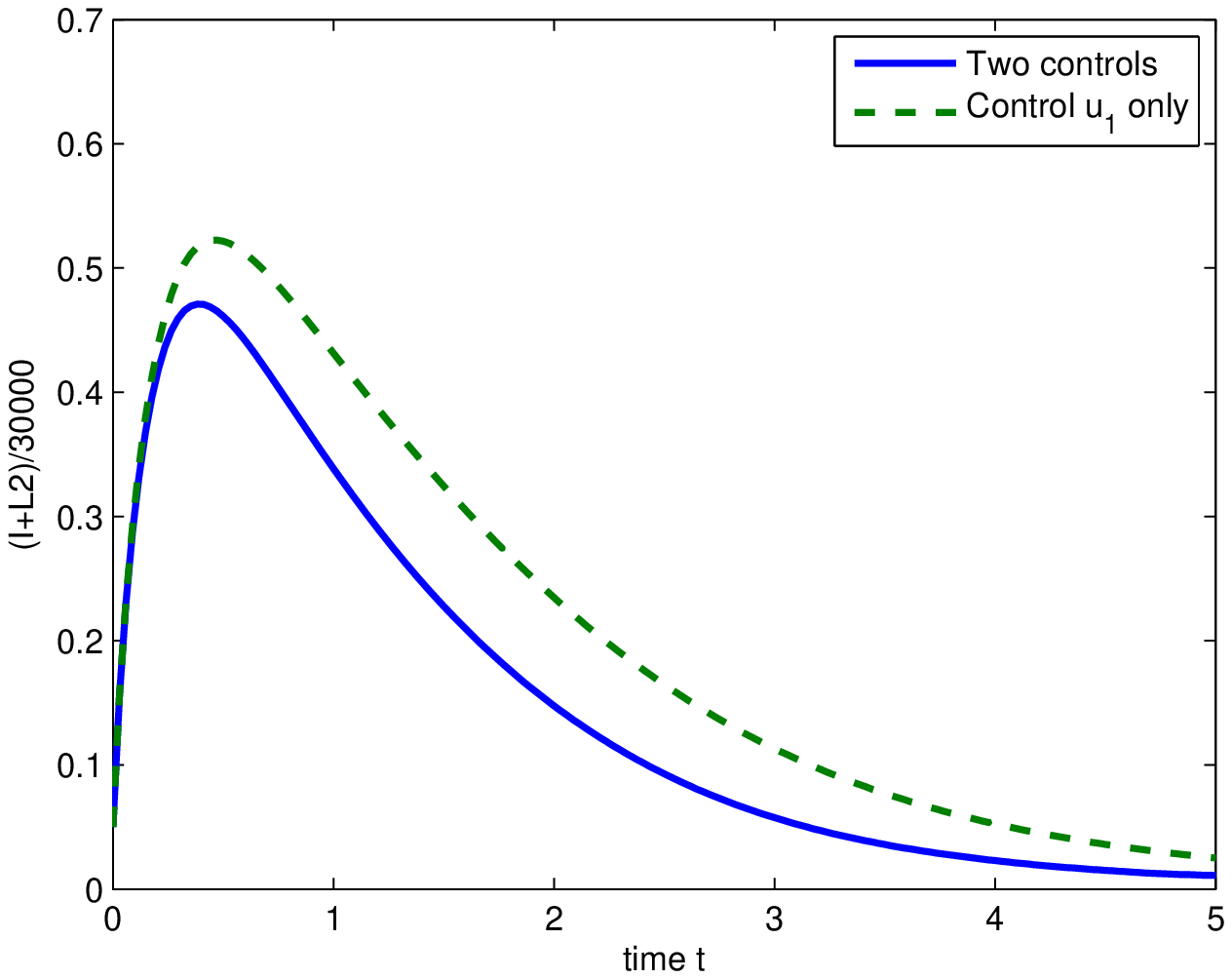}}
\subfloat[\footnotesize{$S^*/N$}]{\label{fig:so:u1:S}
  \includegraphics[width=0.45\textwidth]{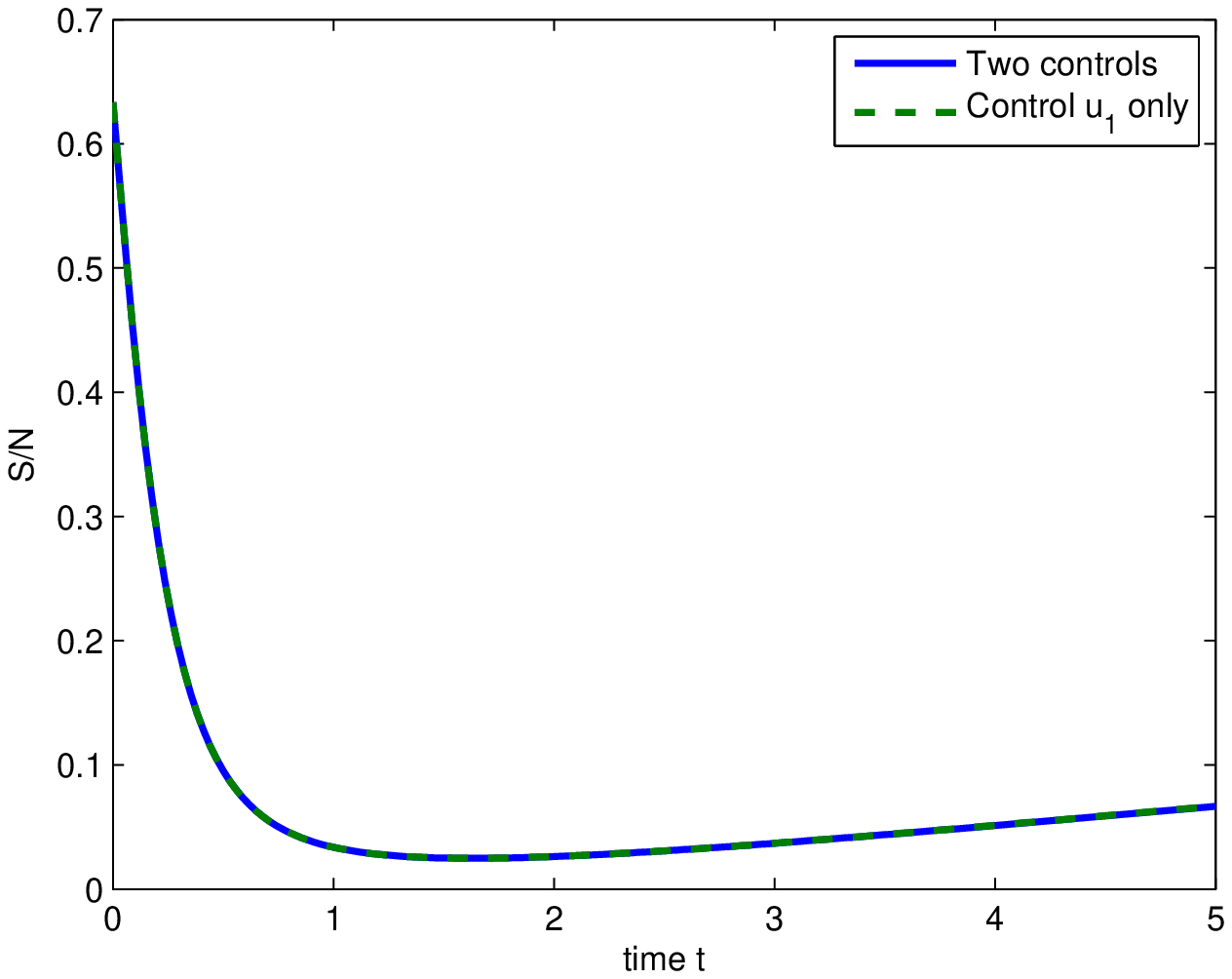}}
\newline
\subfloat[\footnotesize{$L_1^*/N$}]{\label{fig:so:u1:L1}
  \includegraphics[width=0.45\textwidth]{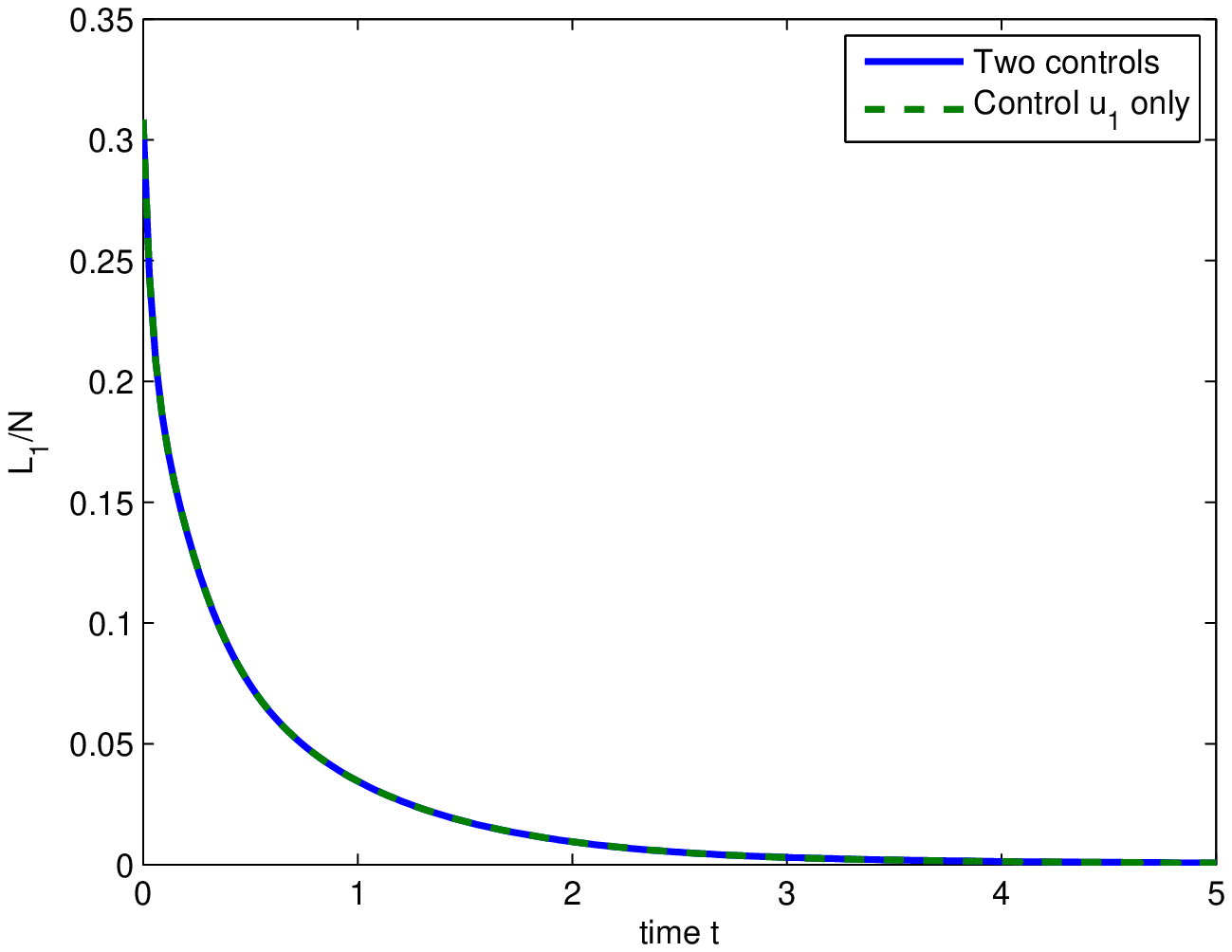}}
\subfloat[\footnotesize{$I^*/N$}]{\label{fig:so:u1:I}
  \includegraphics[width=0.45\textwidth]{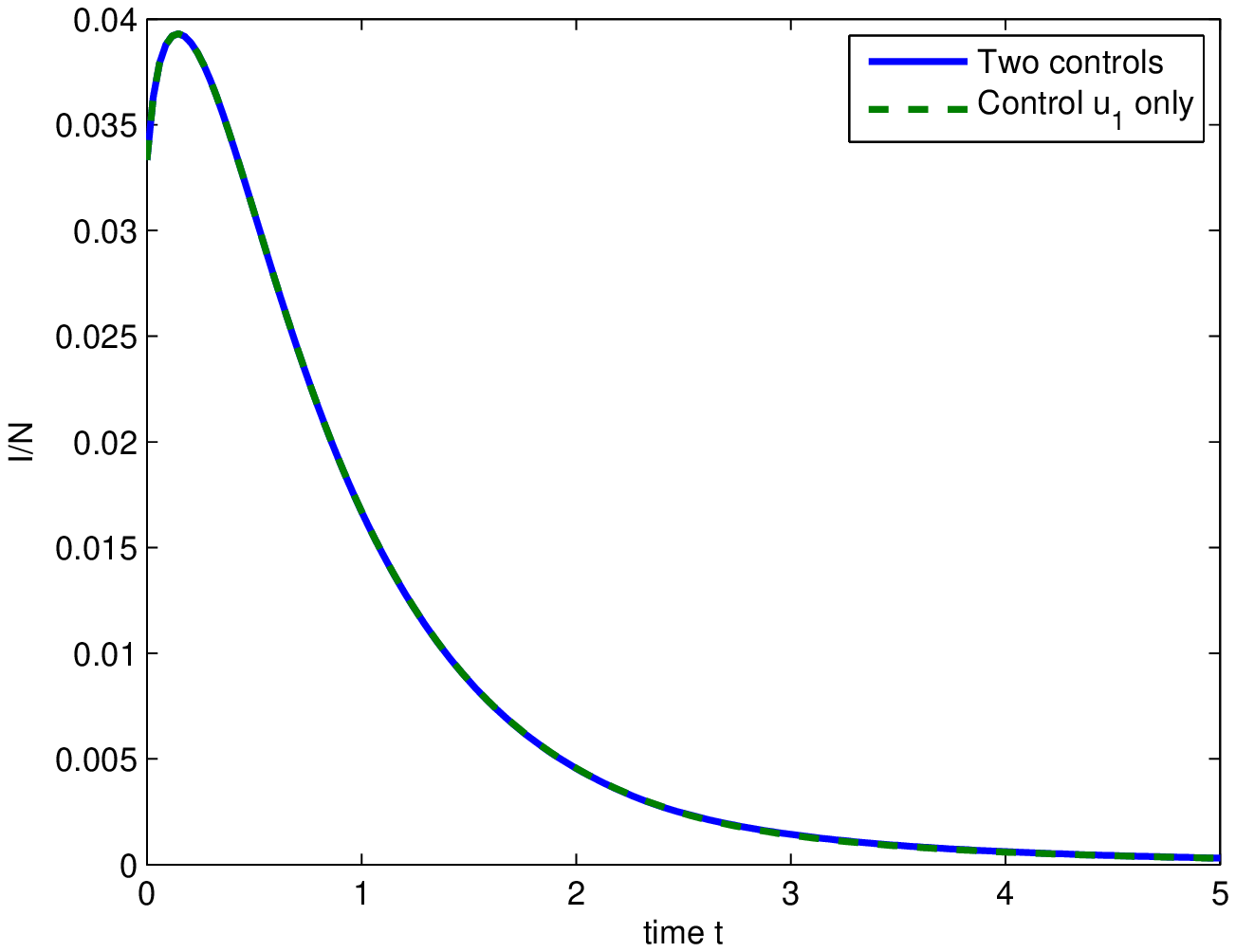}}
\newline
\subfloat[\footnotesize{$L_2^*/N$}]{\label{fig:so:u1:L2}
  \includegraphics[width=0.45\textwidth]{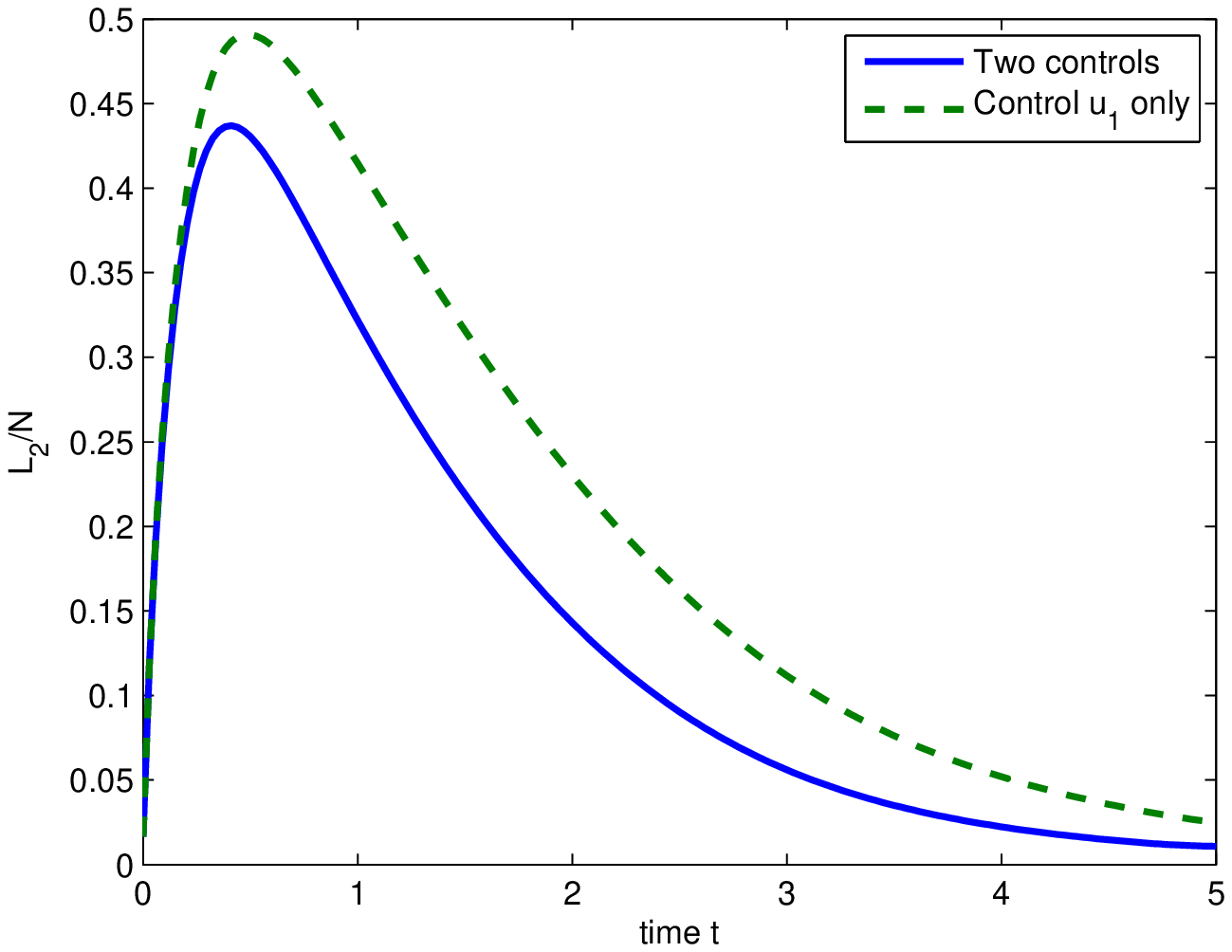}}
\subfloat[\footnotesize{$R^*/N$}]{\label{fig:so:u1:R}
  \includegraphics[width=0.45\textwidth]{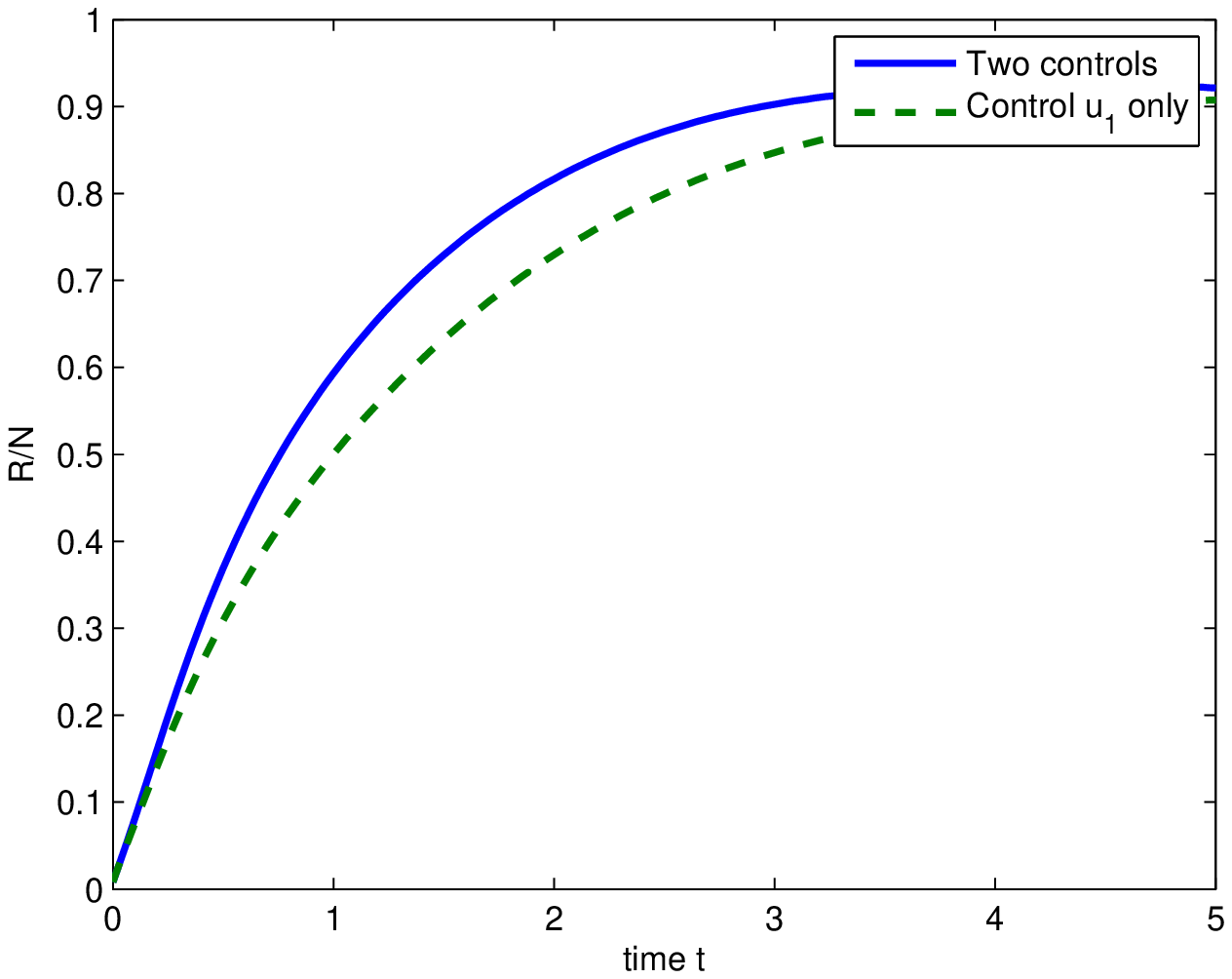}}
\caption{Optimal states for Strategy~1 (dashed line) and Strategy~3 (solid line)
in a time scale in units of years}
\label{fig:so:u1:statevar}
\end{center}
\end{figure}

% -----------------------------------------------------

\subsection{Strategy 2 versus Strategy~3}

When we compare Strategy~2 (only control $u_2$)
with Strategy~3 (both controls $u_1$ and $u_2$),
we observe that the best choice is definitely to use Strategy~3.
Indeed, with Strategy~3, there is a lower fraction of early and persistent
latent individuals as well as infected individuals
(see Figure~\ref{fig:so:u2:statevar} A, C, D and E).
The fraction of susceptible and recovered individuals
is higher when two controls are applied at the same time
(see Figure~\ref{fig:so:u2:statevar} B and F).
Similarly with Strategy~1, the value of the cost functional \eqref{costfunction}
associated with Strategy~3 is lower than that of Strategy~2.
% ------------------------------------------------
\begin{figure}
\begin{center}
\subfloat[\footnotesize{$(I^* + L_2^*)/N$}]{\label{fig:so:u2:I:L2}
  \includegraphics[width=0.45\textwidth]{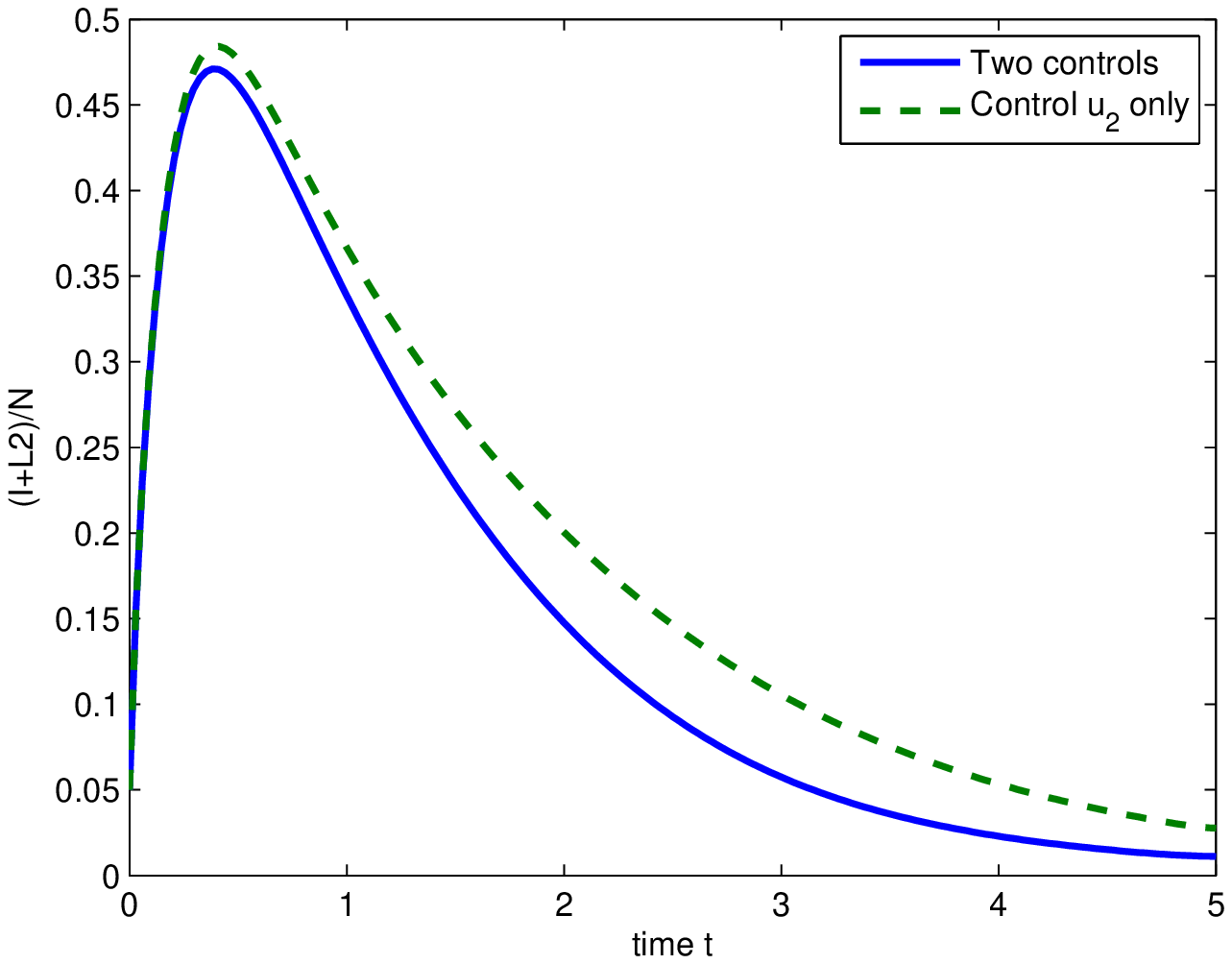}}
\subfloat[\footnotesize{$S^*/N$}]{
  \includegraphics[width=0.45\textwidth]{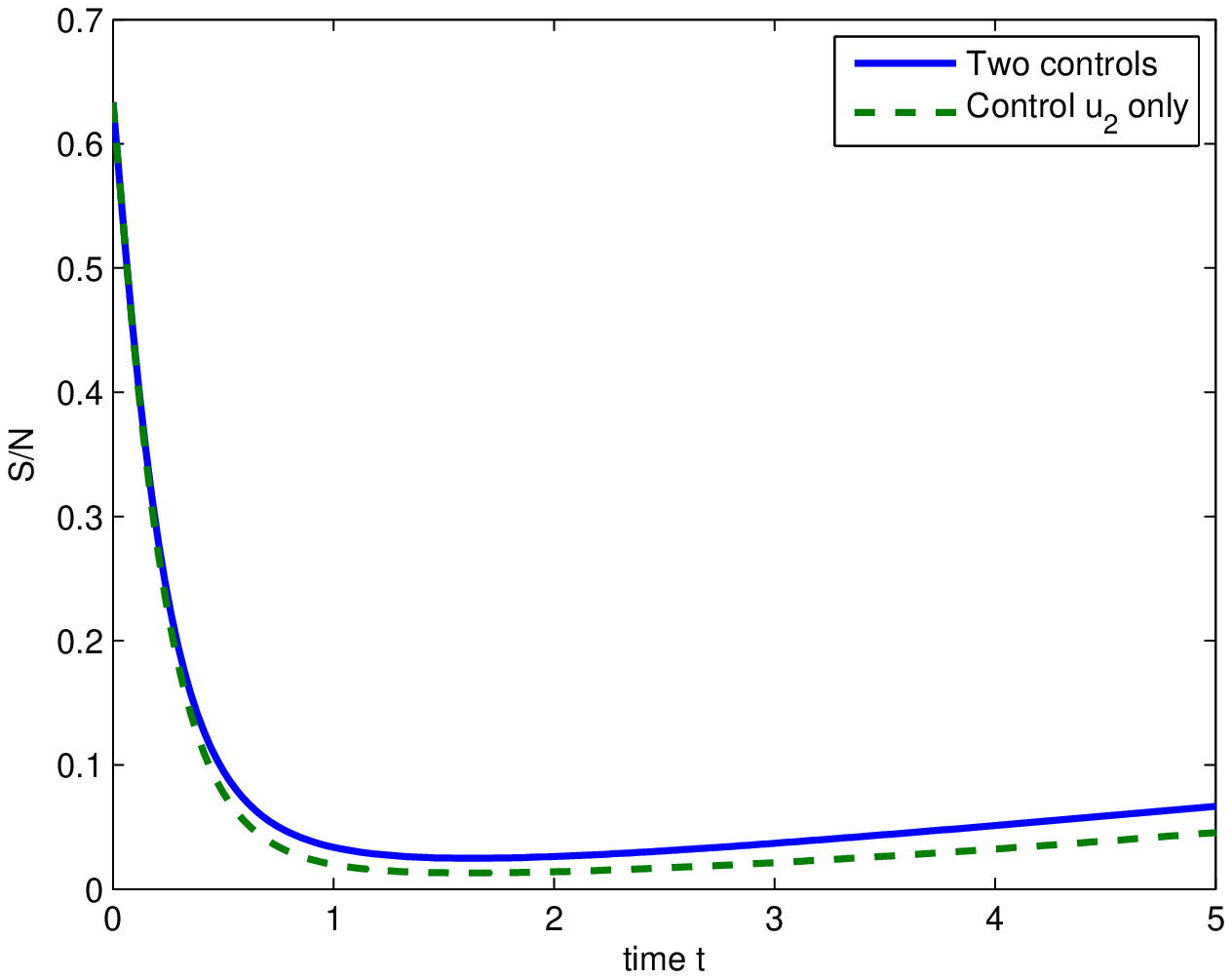}}
\newline
\subfloat[\footnotesize{$L_1^*/N$}]{\label{fig:so:u2:L1}
  \includegraphics[width=0.45\textwidth]{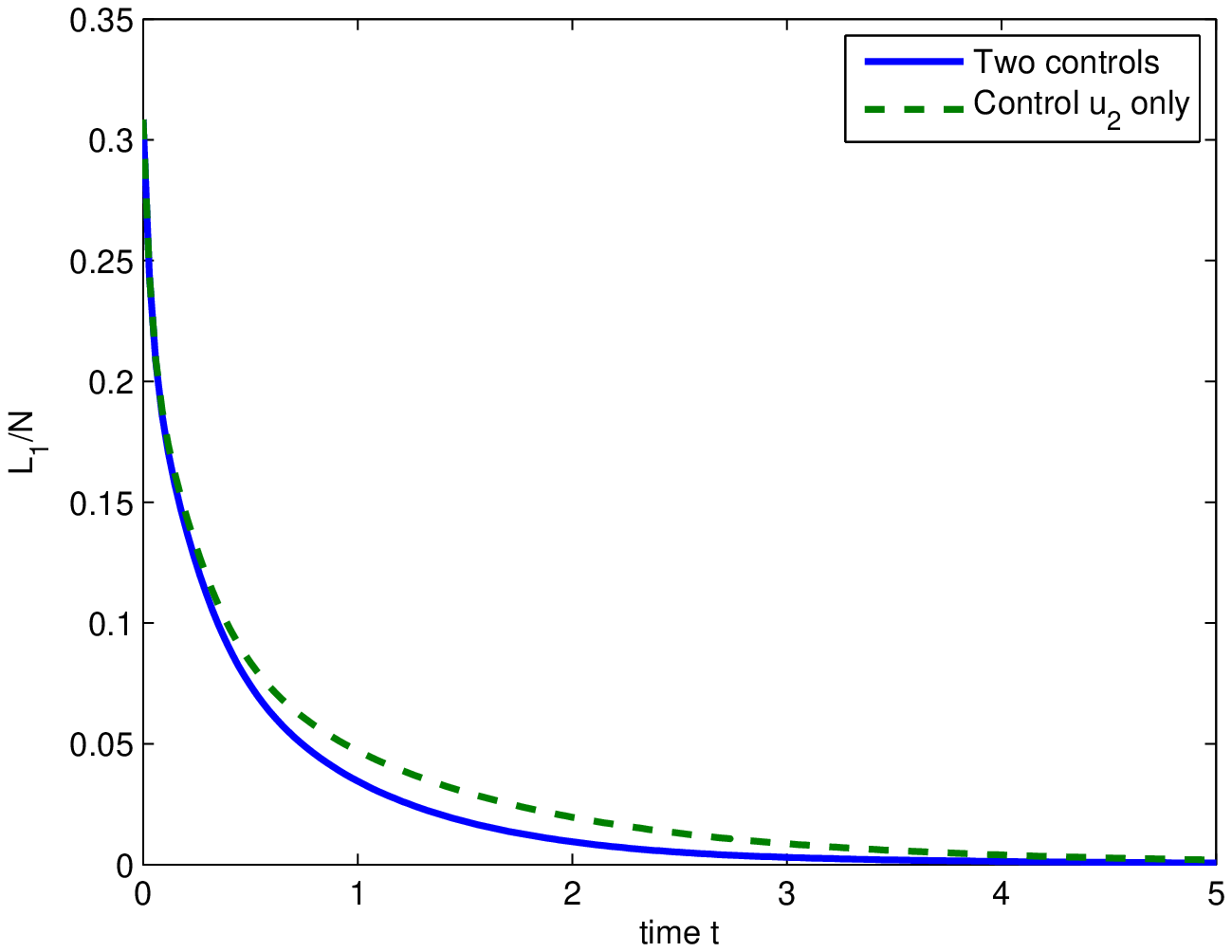}}
\subfloat[\footnotesize{$I^*/N$}]{\label{fig:so:u2:I}
  \includegraphics[width=0.45\textwidth]{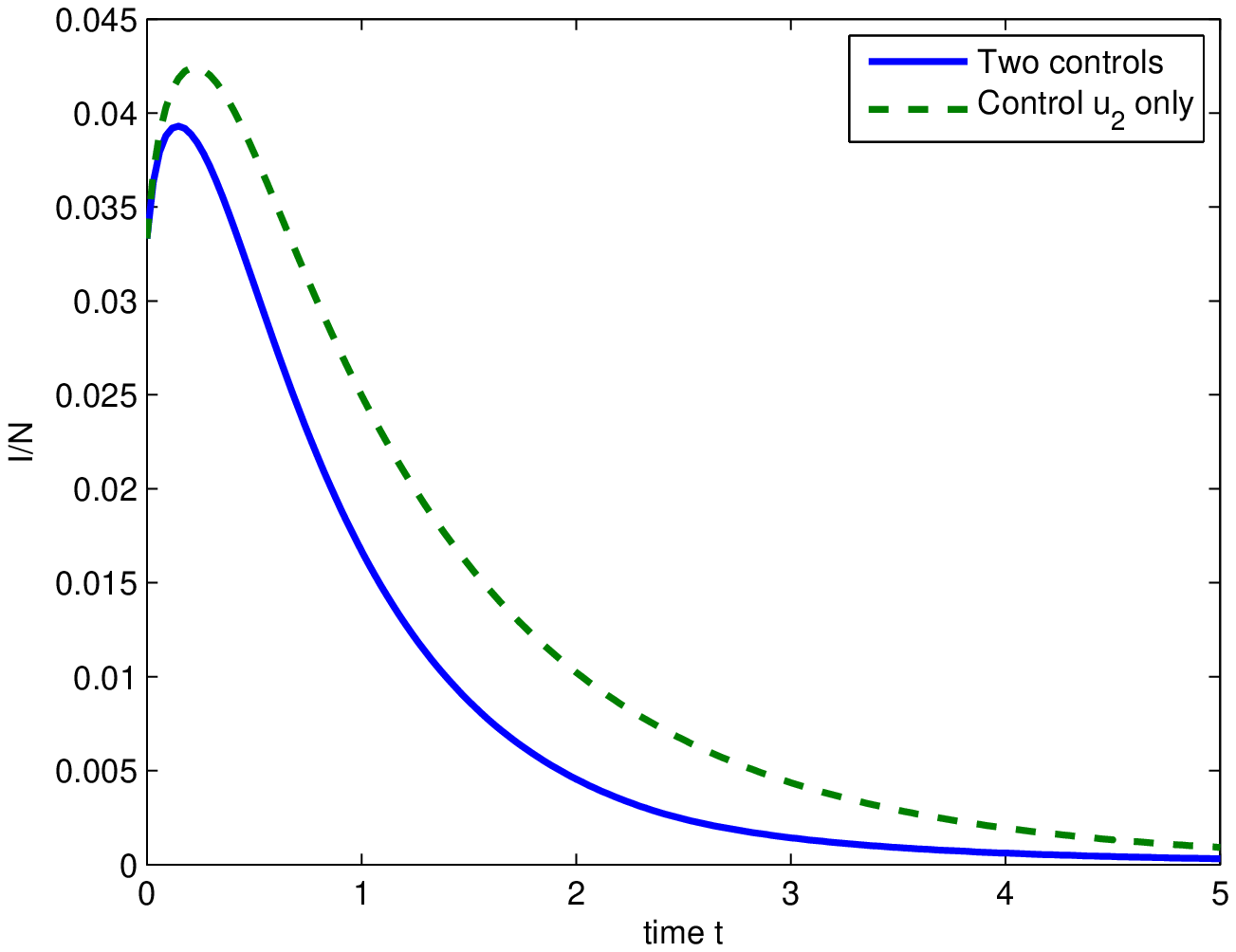}}
\newline
\subfloat[\footnotesize{$L_2^*/N$}]{\label{fig:so:u2:L2}
  \includegraphics[width=0.45\textwidth]{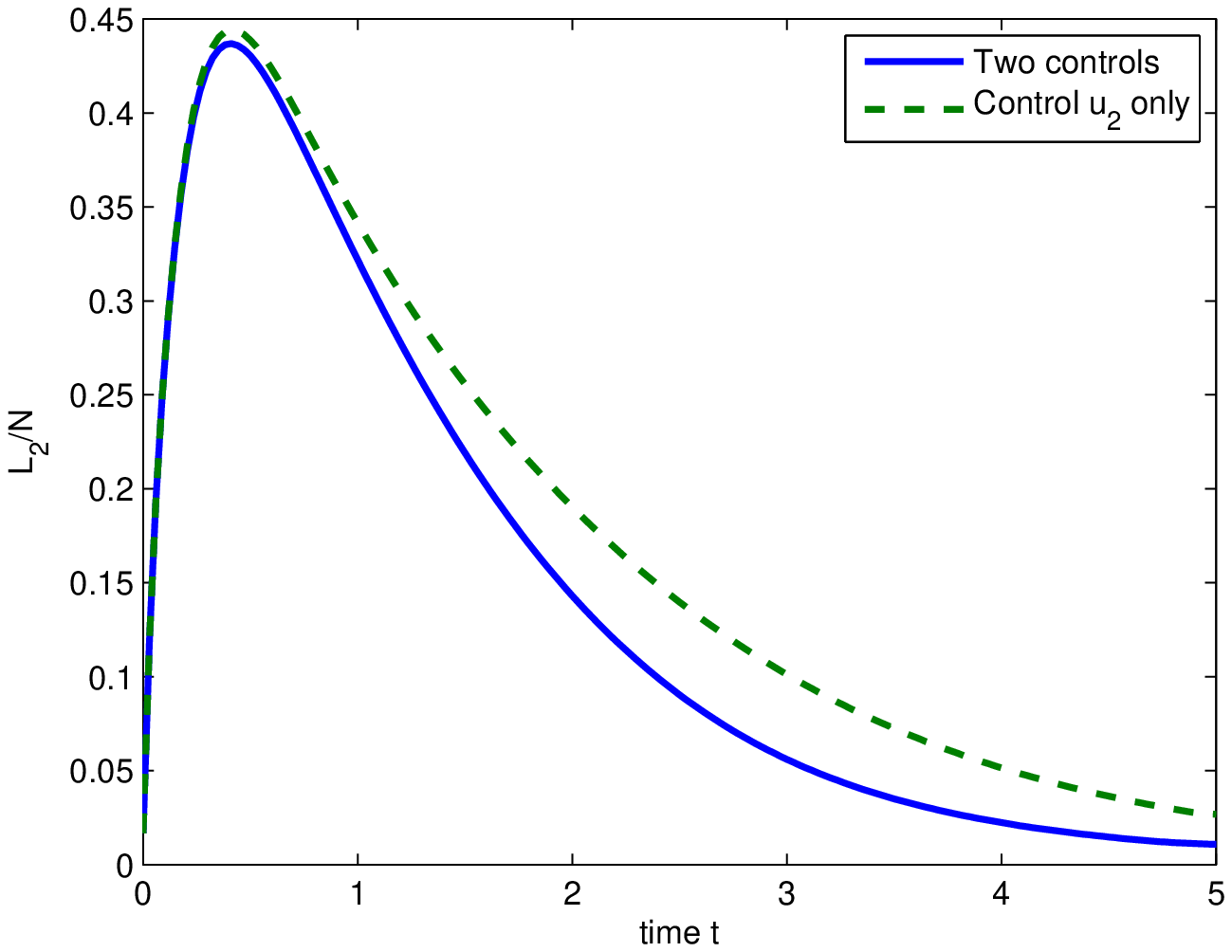}}
\subfloat[\footnotesize{$R^*/N$}]{\label{fig:so:u2:R}
  \includegraphics[width=0.45\textwidth]{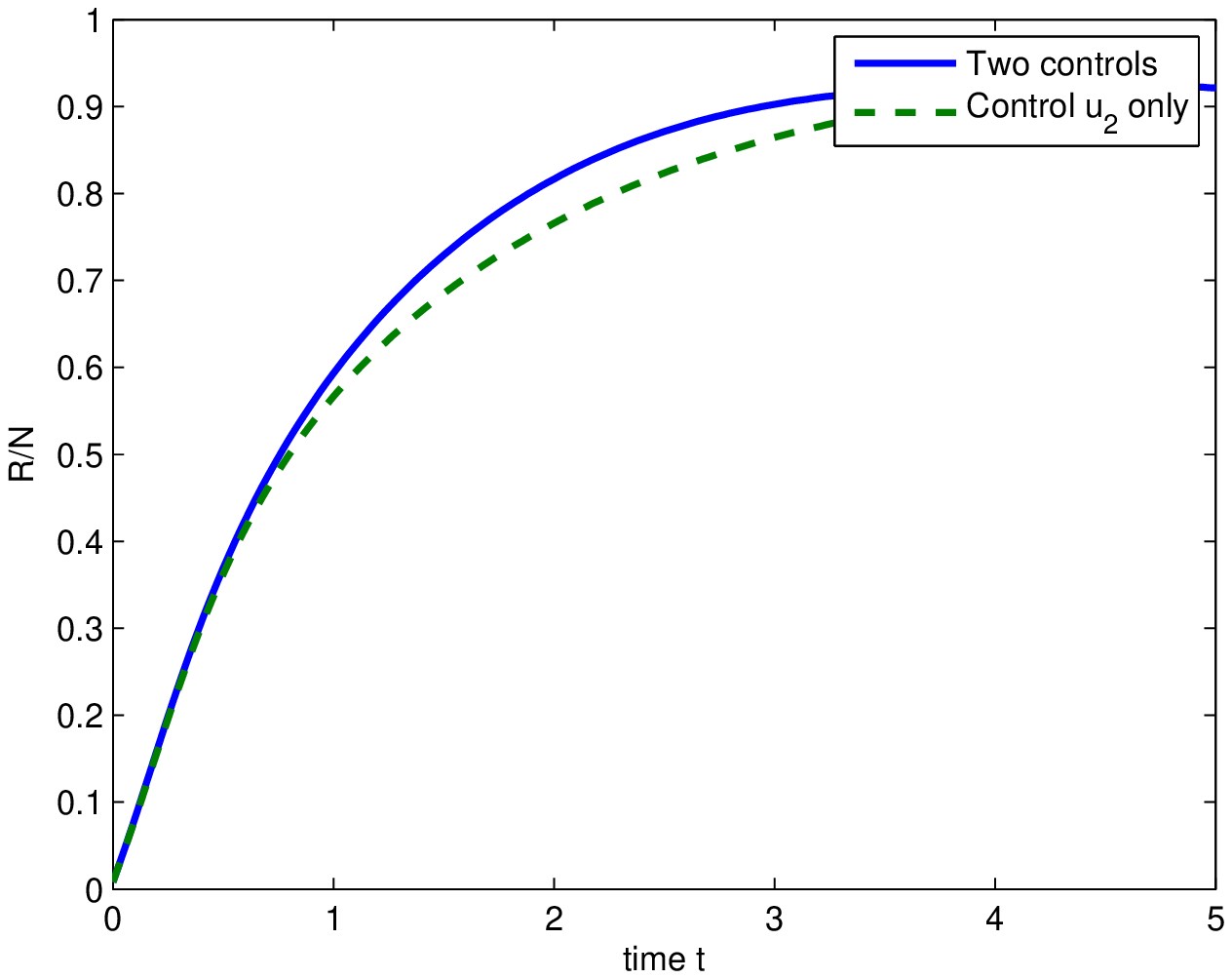}}
\caption{Optimal states for Strategy~2 (dashed line) and Strategy~3 (solid line)
in a time scale in units of years}
\label{fig:so:u2:statevar}
\end{center}
\end{figure}

% -----------------------------------------------------

\subsection{Changing the relative cost of interventions}
\label{sec:same}

So far all the simulations were done
assuming that the weight $W_1$ associated with
the control $u_1$ is greater than the weight $W_2$ associated
with $u_2$: $W_1 = 500$ and $W_2 = 50$ (see Table~\ref{parameters}).
We now consider the case when the weights are the same:
$W_1 = W_2 = 50$. In Figure~\ref{CompareWi:iguais:dif1}
we observe that a higher weight associated with the control $u_1$
implies that the optimal control $u_1^*$ stays
at the upper bound for a smaller duration. Surprisingly,
this change in the control does not result in
a change of the behavior of the state variables (see Figure~\ref{CompareWi:iguais:dif2}).
The results with $W_1 = W_2 = 50$ for Strategy~1
are similar to the ones reported in Section~\ref{sub:sec:st1}
(including the number of days that the optimal control $u_1^*$ stays at the upper bound).
% ------------------------------------------------
\begin{figure}
\begin{center}
\subfloat[\footnotesize{Optimal control $u_1^*$}]{
\includegraphics[width=0.45\textwidth]{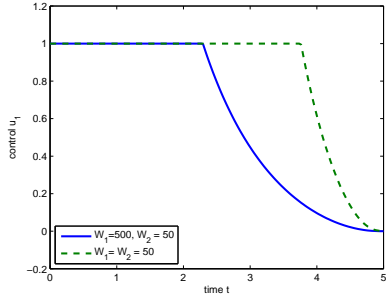}}
\caption{Optimal control $u_1^*$ for Strategy~3
in a time scale in units of years: $W_1=500$ and $W_2 = 50$ (solid line)
versus $W_1 = W_2 = 50$ (dashed line)}
\label{CompareWi:iguais:dif1}
\end{center}
\end{figure}
% ------------------------------------------------
\begin{figure}
\begin{center}
\subfloat[\footnotesize{$(I^*+L_2^*)/N$}]{
  \includegraphics[width=0.45\textwidth]{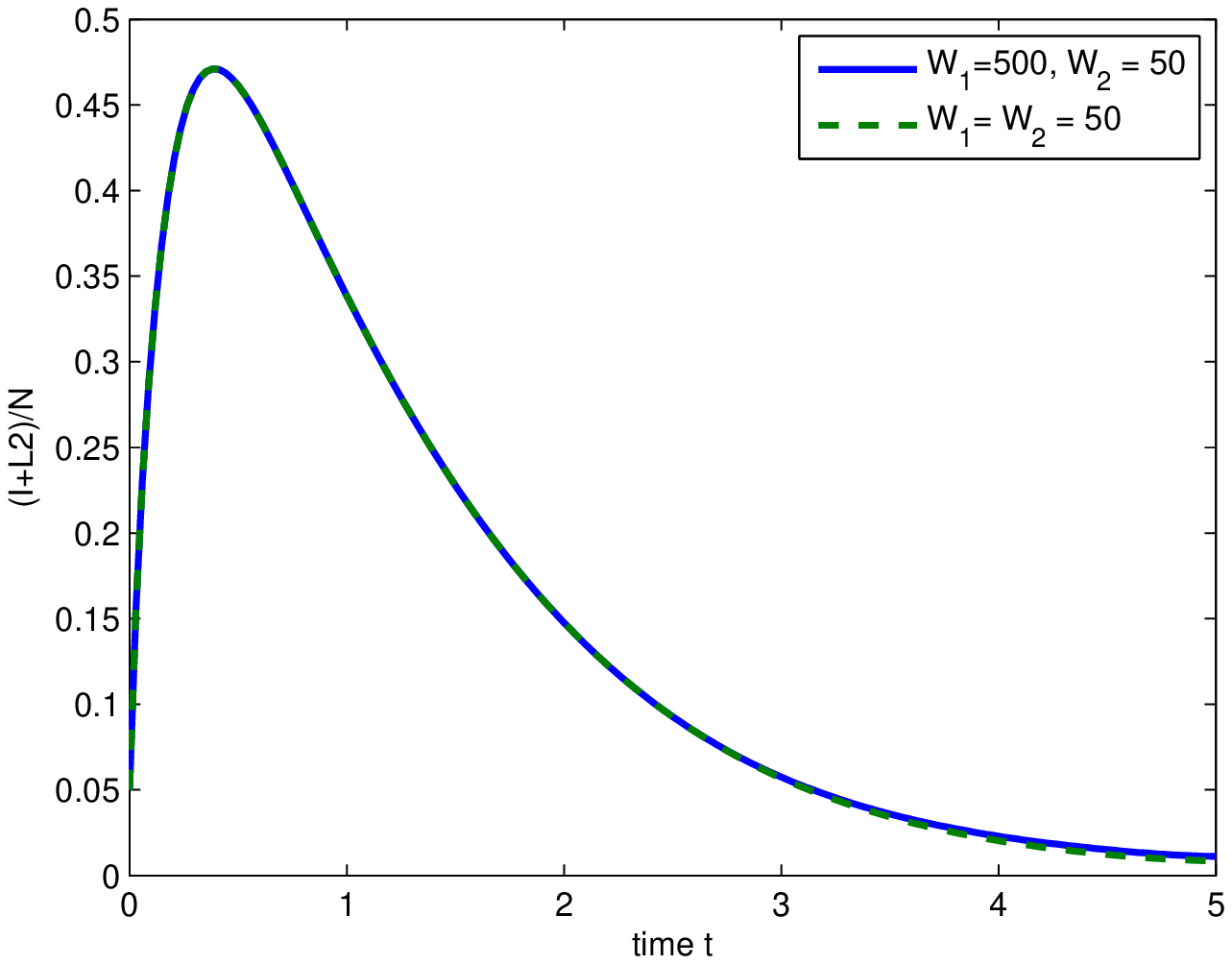}}
\subfloat[\footnotesize{$S^*/N$}]{
  \includegraphics[width=0.45\textwidth]{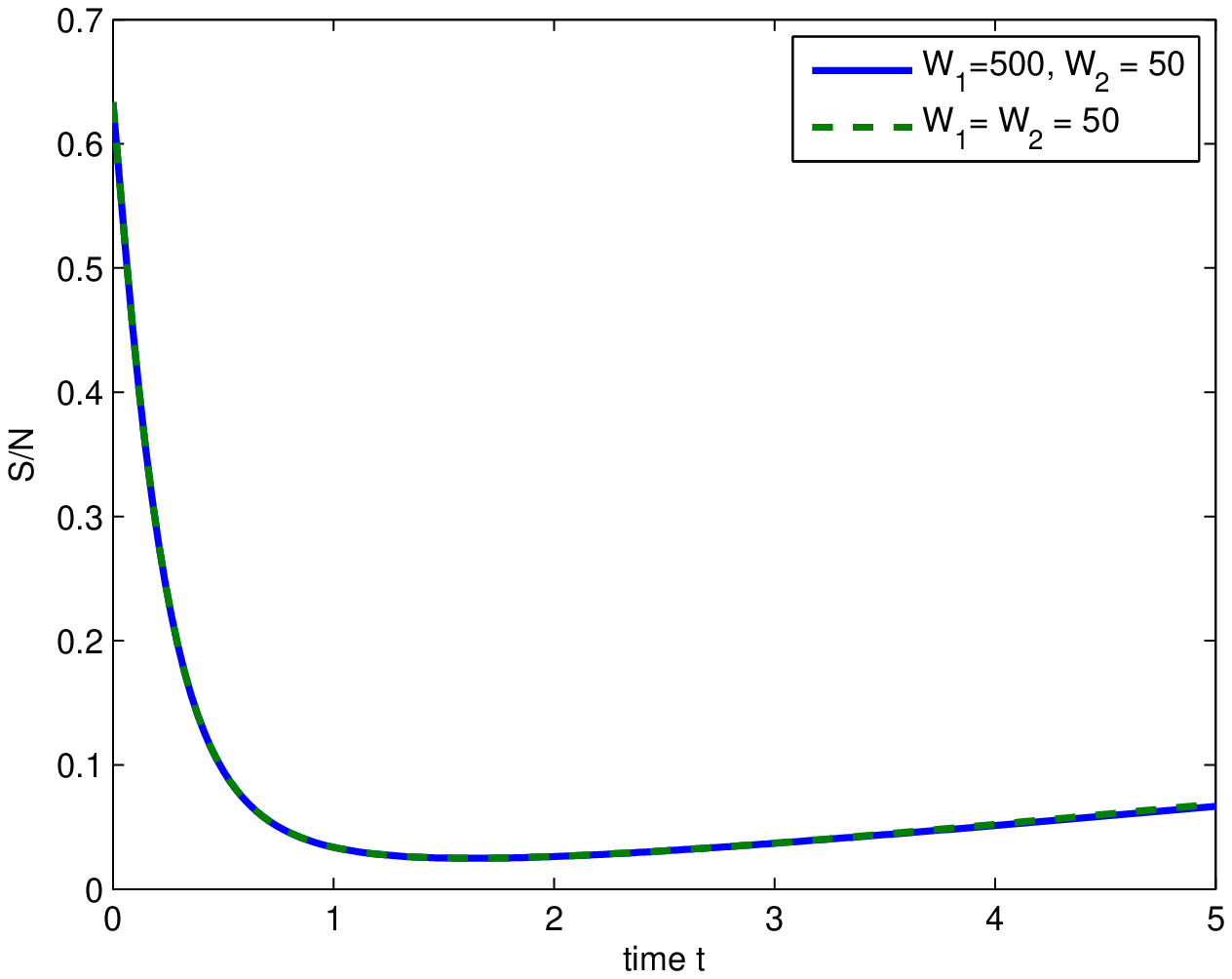}}
\newline
\subfloat[\footnotesize{$L_1^*/N$}]{\label{fig:Wi:igual:Wi:dif:L1}
  \includegraphics[width=0.45\textwidth]{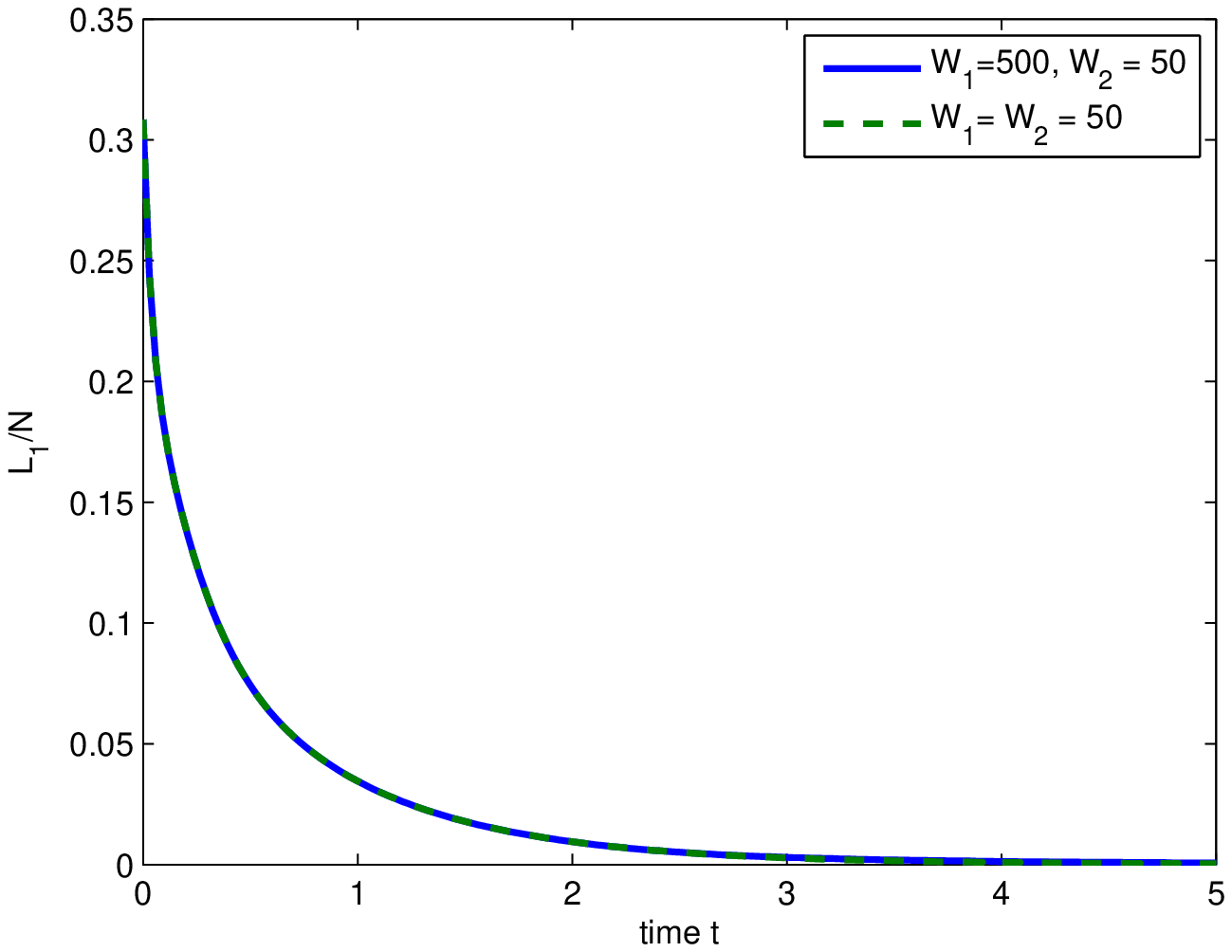}}
\subfloat[\footnotesize{$I^*/N$}]{\label{fig:Wi:igual:Wi:dif:I}
  \includegraphics[width=0.45\textwidth]{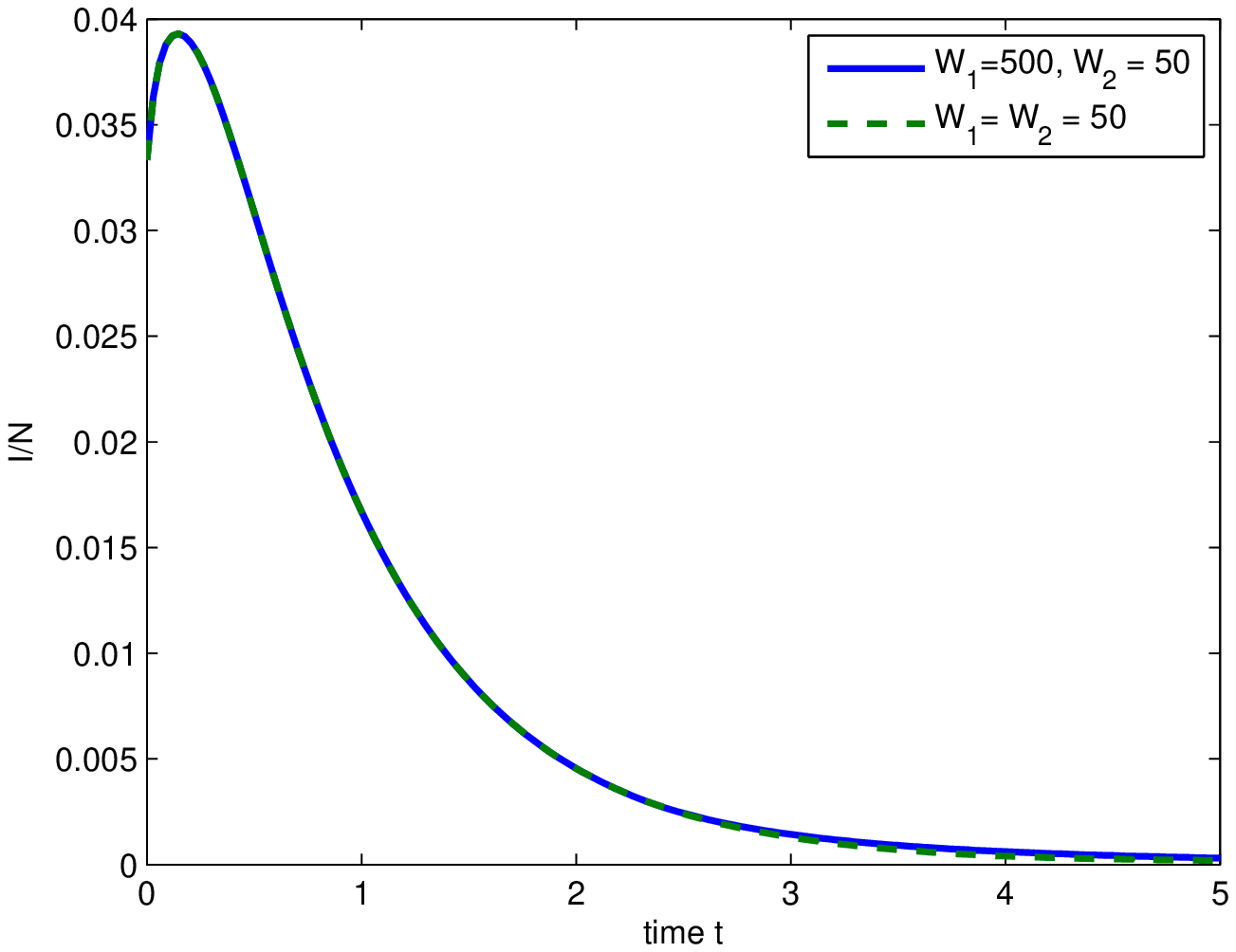}}
\newline
\subfloat[\footnotesize{$L_2^*/N$}]{\label{fig:Wi:igual:Wi:dif:L2}
  \includegraphics[width=0.45\textwidth]{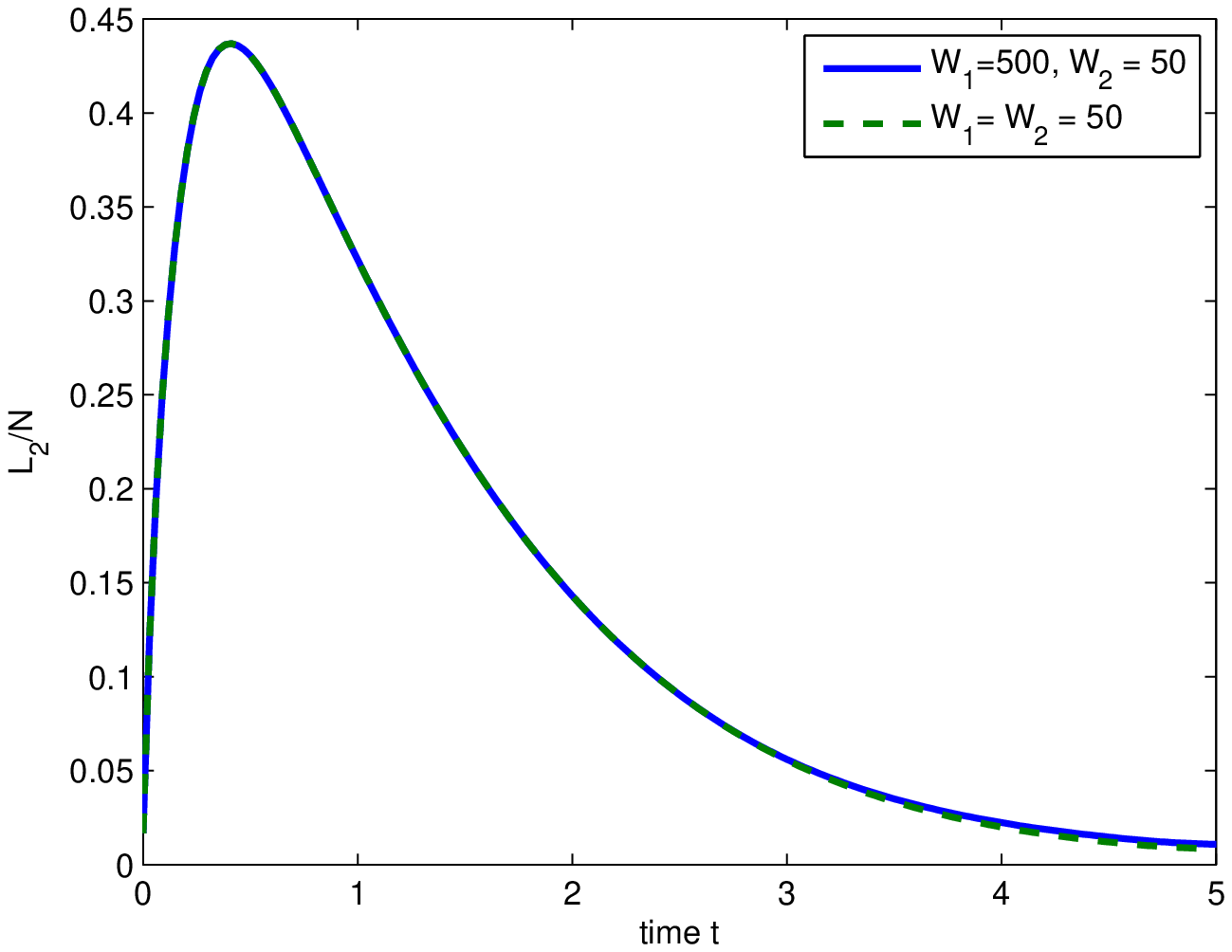}}
\subfloat[\footnotesize{$R^*/N$}]{\label{fig:Wi:igual:Wi:dif:R}
  \includegraphics[width=0.45\textwidth]{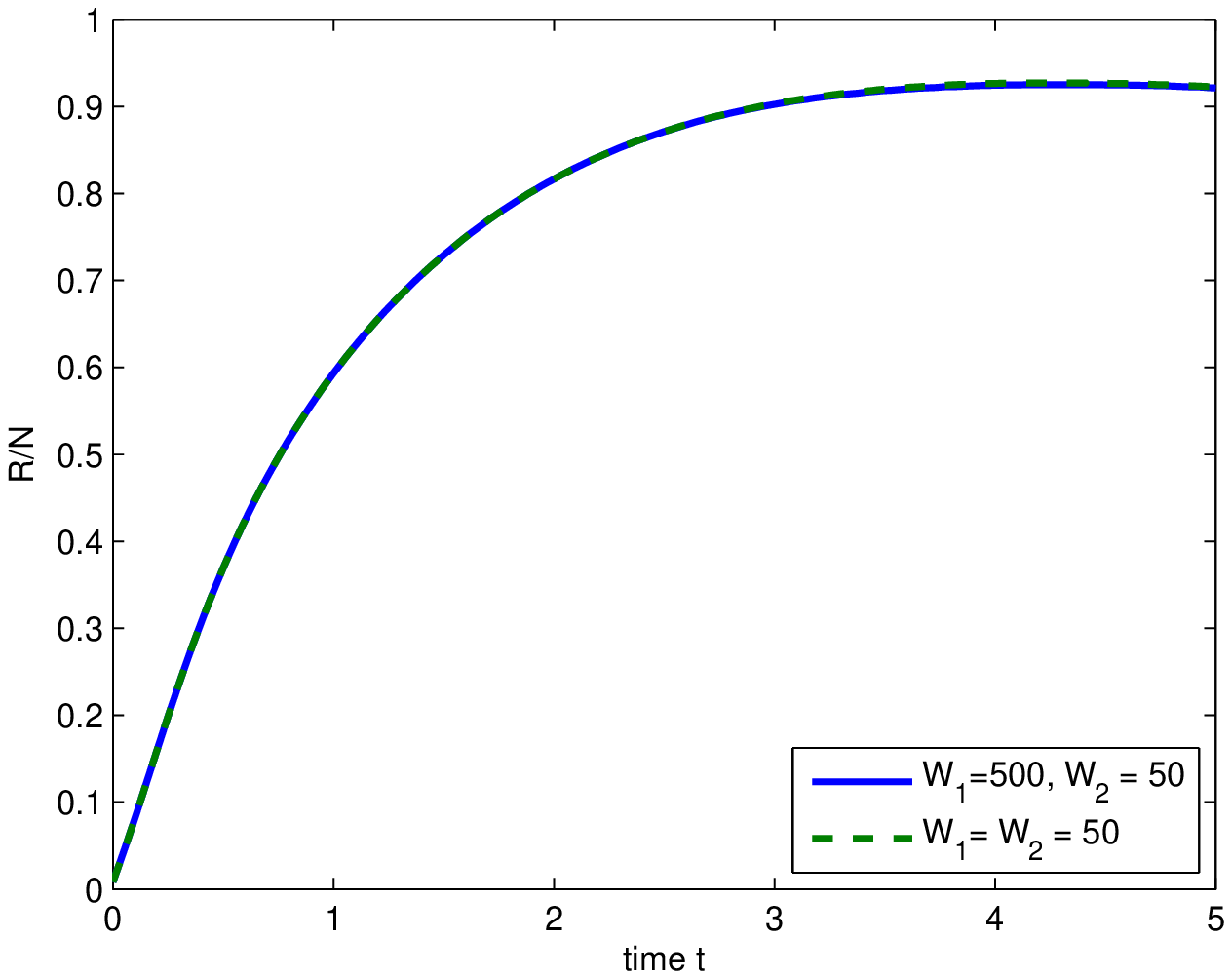}}
\caption{Optimal states for Strategy~3 in a time scale in units of years:
$W_1=500$ and $W_2 = 50$ (solid line) versus $W_1 = W_2 = 50$ (dashed line)}
\label{CompareWi:iguais:dif2}
\end{center}
\end{figure}

% -----------------------------------------------------

\section{Conclusion}
\label{sec:conc}

Tuberculosis (TB) is an important health issue all over the world,
particularly in many African countries.
In this article we focused our attention in Angola where,
since 2009, it is recognized by authorities
that TB and HIV are the two most serious
public health problems. Efforts to control TB in Angola
have been underway since 1981, but results are still
not satisfactory, and minimizing
the effects caused by TB is an important challenge.
Implemented measures to control these problems
are having a positive effect, but the reduction rate
is very slow \cite{url_vice_Min_Saude_2009}.

We introduced two control functions $u_1$ and $u_2$
to an existing mathematical model for TB developed
in \cite{Gomes_etall_2007}. These controls are associated
with measures that help to reduce the number of active infected and persistent
latent TB individuals: the control $u_1$
represents the effort that prevents the failure of treatment
in active TB infectious individuals $I$, \textrm{e.g.}, supervising the patients,
helping them to take the TB medications regularly and to complete the TB treatment;
while the control $u_2$ governs the fraction of persistent latent individuals $L_2$
under treatment with anti-TB drugs. An optimal control problem was formulated
and solved theoretically using the Pontryagin maximum principle.
The solution to the problem was then illustrated by numerical simulations
using available data from Angola. From the numerical results,
one may conclude that considering only control $u_1$ (Strategy~1)
or only control $u_2$ (Strategy~2) does not lead to the best results
in terms of the number of active infected
and persistent latent individuals. A combined strategy
(Strategy~3) that involves both controls is preferable.
Figure~\ref{fig:without:controls:2} shows the simulation
of the system for the uncontrolled case,
when $u_1 = 0$ and $u_2 = 0$, compared with the strategies
here proposed. Results justify the need for intervention in tuberculosis treatment.
The values of Table~\ref{without:controls} put in evidence that
the uncontrolled situation is the worst: does not only result in
more infected and persistent latent individuals
(1,584 individuals at the end of five years versus
334 individuals in the case Strategy~3 is applied) but also
gives a higher value to the cost functional \eqref{costfunction}.
% -----------------------------------------------------
\begin{figure}
\begin{center}
\subfloat[\footnotesize{$(I+L_2)/N$}]{
  \includegraphics[width=0.45\textwidth]{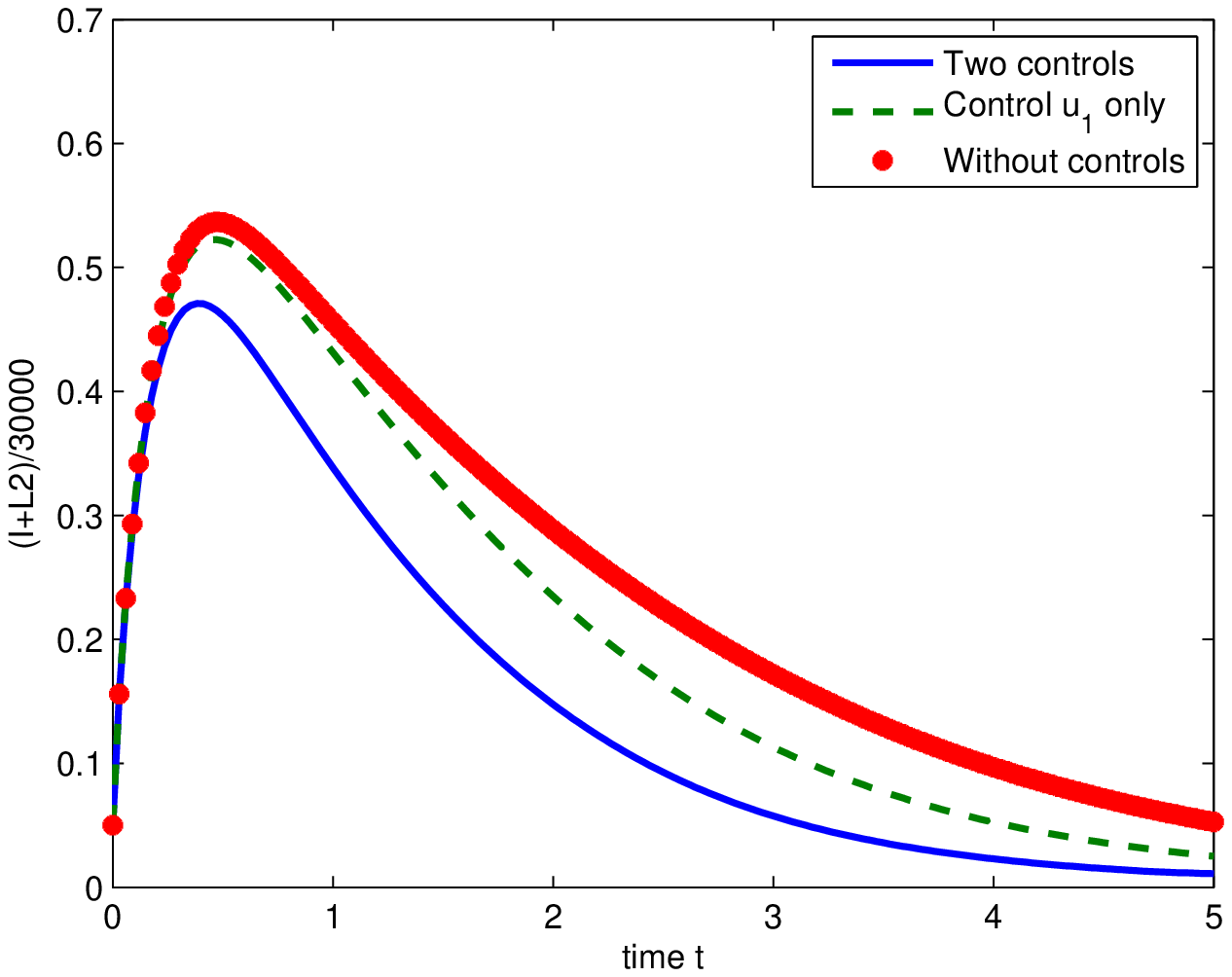}}
\subfloat[\footnotesize{$S/N$}]{
  \includegraphics[width=0.45\textwidth]{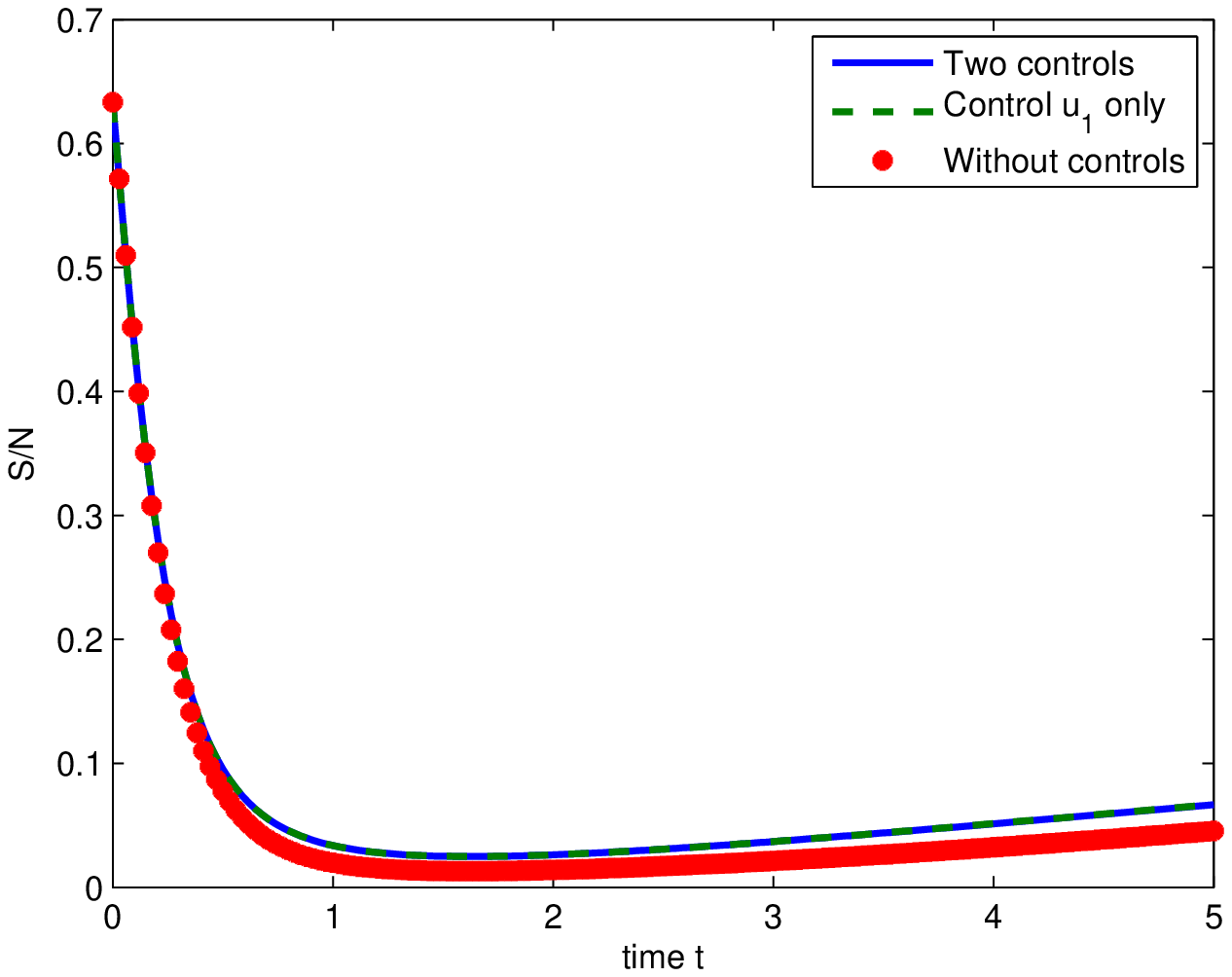}}
\newline
\subfloat[\footnotesize{$L_1/N$}]{
  \includegraphics[width=0.45\textwidth]{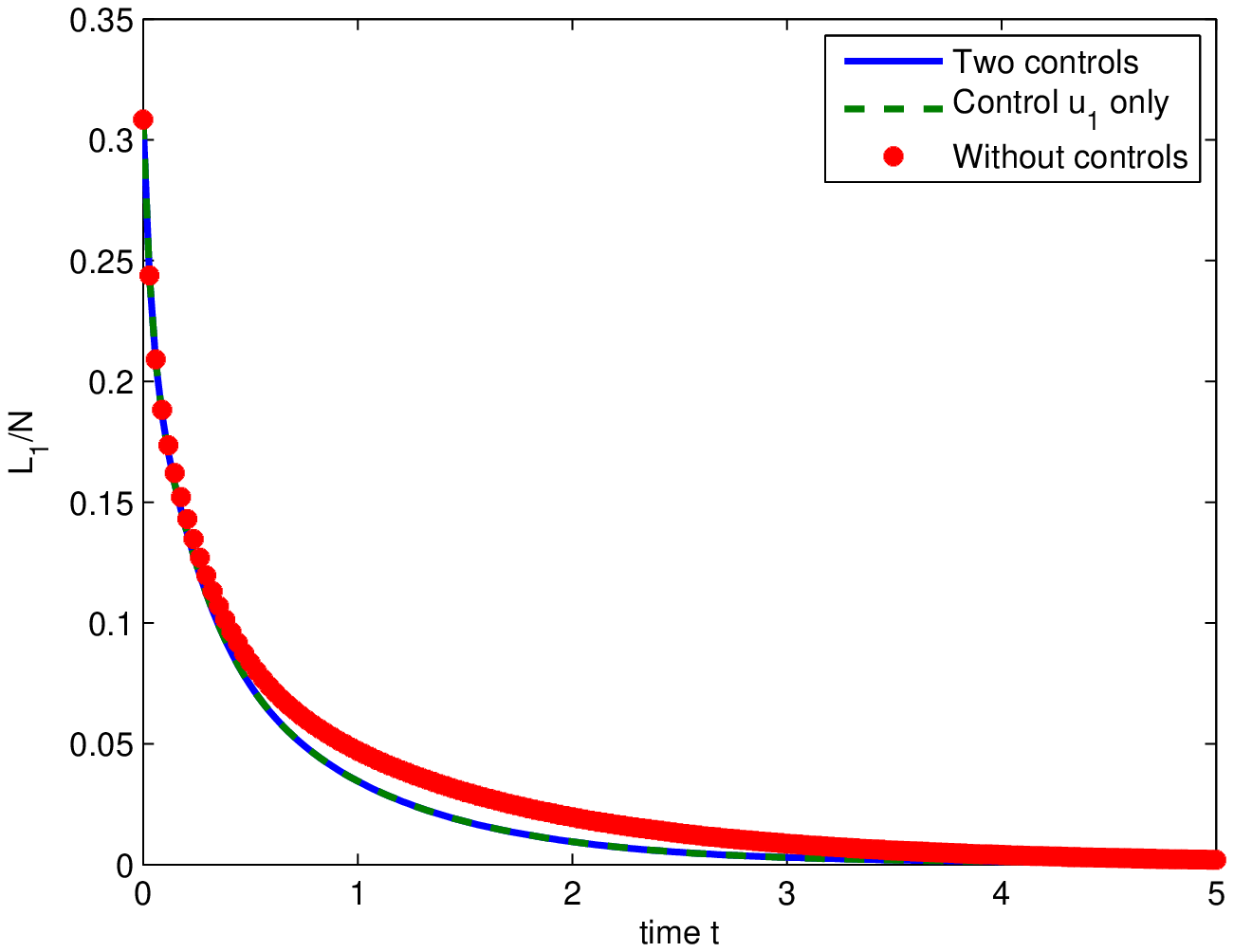}}
\subfloat[\footnotesize{$I/N$}]{
  \includegraphics[width=0.45\textwidth]{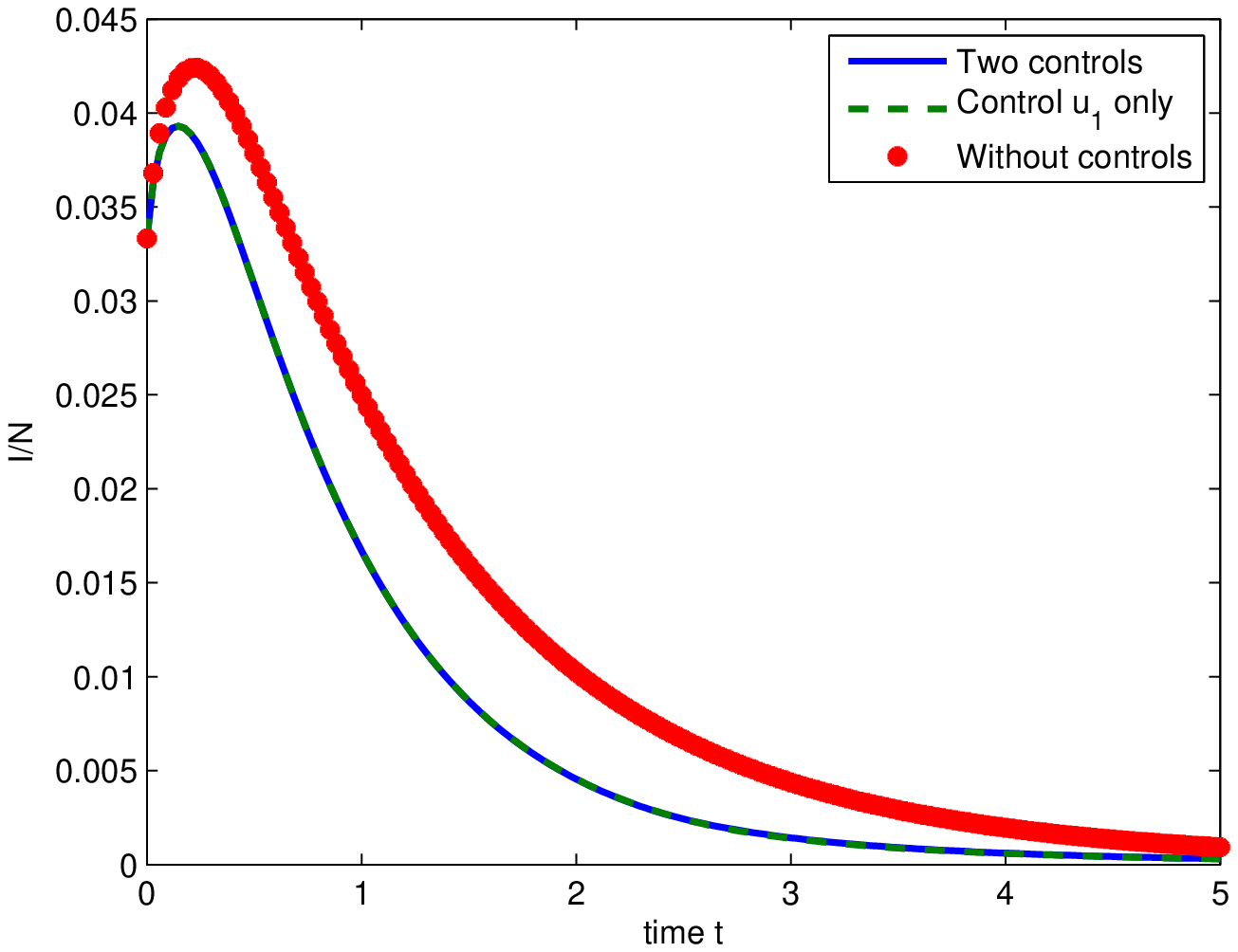}}
\newline
\subfloat[\footnotesize{$L_2/N$}]{
  \includegraphics[width=0.45\textwidth]{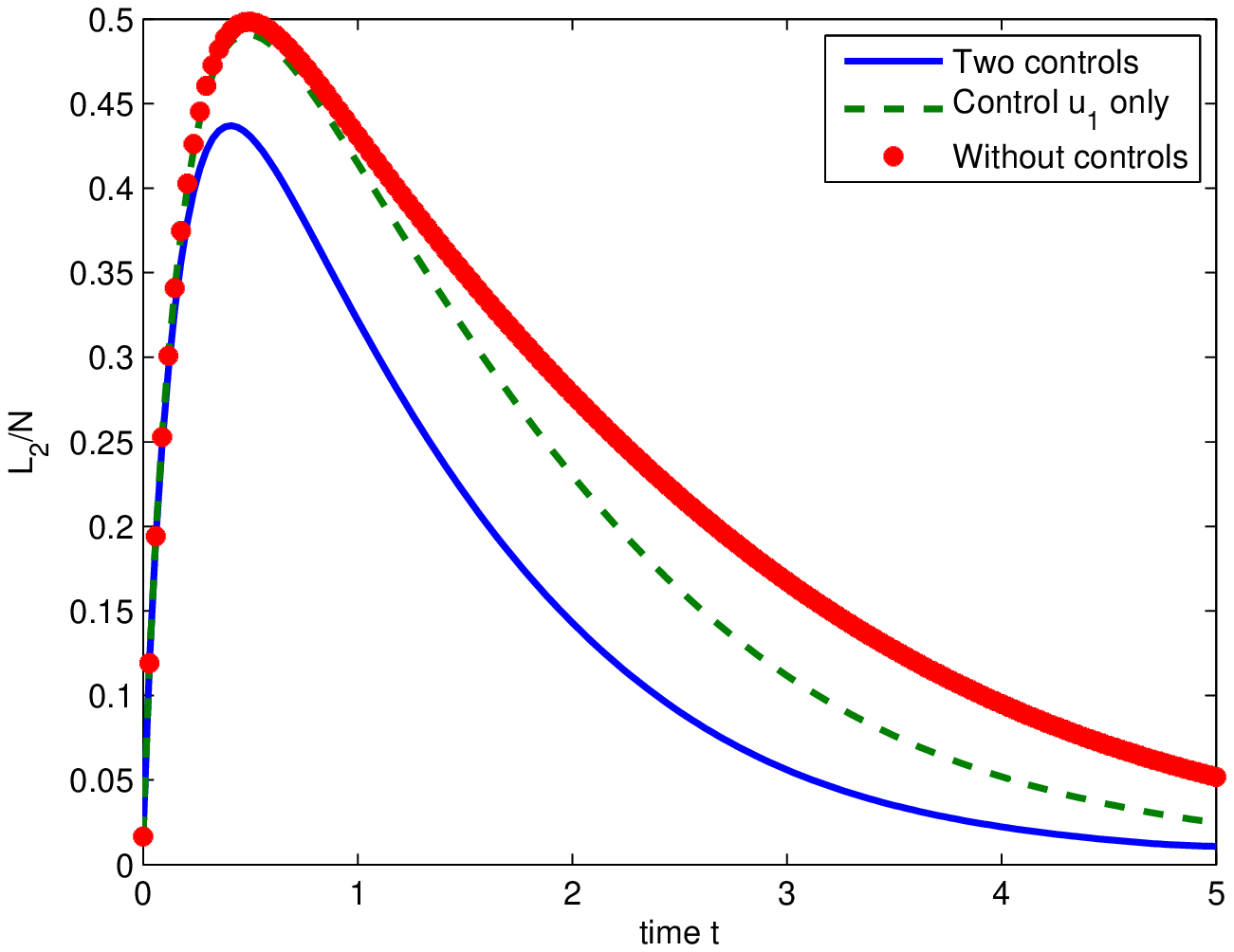}}
\subfloat[\footnotesize{$R/N$}]{
  \includegraphics[width=0.45\textwidth]{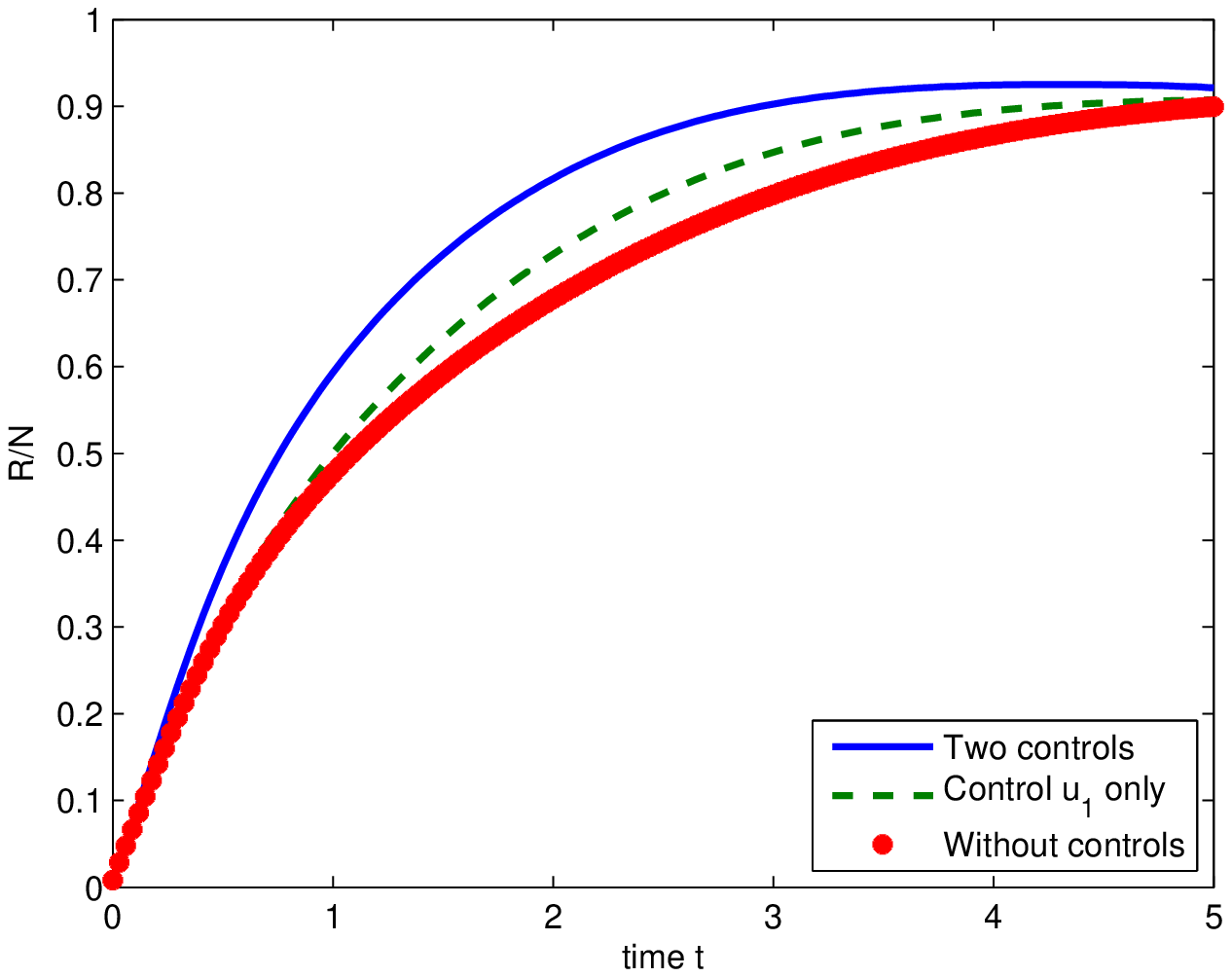}}
\caption{State variables with and without controls in a time scale in units of years}
\label{fig:without:controls:2}
\end{center}
\end{figure}
% -----------------------------------------------------
\begin{table}[!htb]
\centering
\begin{tabular}{|l | l | l  | l | l |}
\hline
 & {\small{Strategy 1}} & {\small{Strategy 2}} & {\small{Strategy 3}} & {\small{Without controls}}\\
\hline
{\small{$I(5)+L_2(5)$}} & {\small{$\simeq 751$}} & {\small{$\simeq 831$}} & {\small{$\simeq 334$}} & {\small{$\simeq 1,584$}} \\
\hline
{\small{Cost: functional \eqref{costfunction}}} & {\small{$\simeq 32,511.2$}}
& {\small{$\simeq 28,585.9$}} & {\small{$\simeq 24,133.1$}} & {\small{$\simeq 37,760.3$}} \\
\hline
\end{tabular}
\caption{Comparison of Strategies 1, 2, 3 with the uncontrolled case,
\textrm{i.e.}, when $u_1 = u_2 = 0$}
\label{without:controls}
\end{table}

% -----------------------------------------------------

\section*{Acknowledgments}

This work was supported by FEDER funds through COMPETE
(Operational Programme Factors of Competitiveness)
and by Portuguese funds through the Center for Research and Development
in Mathematics and Applications (University of Aveiro) and the Portuguese Foundation
for Science and Technology (FCT), within project PEst-C/MAT/UI4106/2011
with COMPETE number FCOMP-01-0124-FEDER-022690.

The authors are grateful to Ryan Loxton
for sharing with them his expertise on
the MISER optimal control software \cite{miser,MR2527217},
and suggestions for improving the English;
and to two anonymous referees for valuable and prompt remarks
and comments, which significantly contributed to the quality of the paper.

% -----------------------------------------------------

% -----------------------------------------------------

\medskip

% -----------------------------------------------------

\end{document}